\documentclass[a4paper,10pt, oneside]{article}
\usepackage{amsmath,amssymb,amsthm,bm,bbm,mathrsfs,mathtools,graphicx}
\usepackage[a4paper,paperheight=297mm,paperwidth=210mm,margin=1.25in]{geometry} % 设置边距为1英寸  
\usepackage{algorithm,algorithmic}
\usepackage{authblk}
\usepackage[font=small,labelfont=md,textfont=it]{caption}
\usepackage[titletoc, title]{appendix}
\usepackage{enumitem}
\usepackage[colorlinks,linkcolor=black,citecolor=black]{hyperref}
\usepackage{etoolbox}
\usepackage{longtable}
\usepackage{diagbox}
\usepackage{booktabs,makecell,multirow}
\usepackage[capitalise,nameinlink]{cleveref}
\usepackage{cases,color}

\Crefname{figure}{Figure}{Figures}
\crefformat{equation}{\textup{#2(#1)#3}}
\crefrangeformat{equation}{\textup{#3(#1)#4--#5(#2)#6}}
\crefmultiformat{equation}{\textup{#2(#1)#3}}{ and \textup{#2(#1)#3}}
{, \textup{#2(#1)#3}}{, and \textup{#2(#1)#3}}
\crefrangemultiformat{equation}{\textup{#3(#1)#4--#5(#2)#6}}%
{ and \textup{#3(#1)#4--#5(#2)#6}}{, \textup{#3(#1)#4--#5(#2)#6}}{, and \textup{#3(#1)#4--#5(#2)#6}}
\Crefformat{equation}{#2Equation~\textup{(#1)}#3}
\Crefrangeformat{equation}{Equations~\textup{#3(#1)#4--#5(#2)#6}}
\Crefmultiformat{equation}{Equations~\textup{#2(#1)#3}}{ and \textup{#2(#1)#3}}
{, \textup{#2(#1)#3}}{, and \textup{#2(#1)#3}}
\Crefrangemultiformat{equation}{Equations~\textup{#3(#1)#4--#5(#2)#6}}%
{ and \textup{#3(#1)#4--#5(#2)#6}}{, \textup{#3(#1)#4--#5(#2)#6}}{, and \textup{#3(#1)#4--#5(#2)#6}}

\apptocmd{\sloppy}{\hbadness 10000\relax}{}{}

\newcommand{\dual}[1]{\langle {#1} \rangle}

\newcommand{\norm}[1]{\lVert {#1} \rVert}

\newcommand{\abs}[1]{\lvert {#1} \rvert}

\newcommand{\absB}[1]{\Big\lvert {#1} \Big\rvert}
\newcommand{\ssnm}[1]
{
  \left\vert\kern-0.25ex
  \left\vert\kern-0.25ex
  \left\vert
  {#1}
  \right\vert\kern-0.25ex
  \right\vert\kern-0.25ex
  \right\vert
}

\makeatletter
\def\spher@harm#1{%
  \vbox{\hbox{%
      \offinterlineskip
      \valign{&\hb@xt@2\p@{\hss$##$\hss}\vskip.2ex\cr#1\crcr}%
    }\vskip-.36ex}%
}
\def\gshone{\spher@harm{.}}
\def\gshtwo{\spher@harm{.&.}}
\def\gshthree{\spher@harm{.&.&.}}
\let\gsh\spher@harm
\makeatother

\newtheorem{lemma}{Lemma}[section]
\newtheorem{remark}{Remark}[section]
\newtheorem{theorem}{Theorem}[section]

\numberwithin{equation}{section}

\makeatletter\def\@captype{table}\makeatother

\begin{document}
\title{
  \Large\bf Convergence of a spatial semidiscretization for a
  three-dimensional stochastic Allen-Cahn equation with
  multiplicative noise\thanks{This work was partially supported by National Natural
    Science Foundation of China under grant 12301525 and
    Natural Science Foundation of Sichuan Province under grant 2023NSFSC1324.}
}
% \author{Binjie Li\thanks{libinjie@scu.edu.cn} }
% \author{Qin Zhou\thanks{Corresponding author: }}

\author[1]{Qin Zhou\thanks{zqmath@cwnu.edu.cn}}
\author[2]{Binjie Li\thanks{libinjie@scu.edu.cn}}
% \author{Binjie Li\thanks{libinjie@scu.edu.cn}}
% \author[2]{Qin Zhou\thanks{Corresponding author: zhouqinmath@stu.scu.edu.cn}}
\affil[1]{School of Mathematics, China West Normal University, Nanchong 637002, China}
\affil[2]{School of Mathematics, Sichuan University, Chengdu 610064, China}

% \affil{School of Mathematics, Sichuan University, Chengdu 610064, China}

% \affil{School of Mathematics, Sichuan University, Chengdu 610064, China}

% \author*[1]{\fnm{First} \sur{Author}}\email{iauthor@gmail.com}



% \affil*[1]{\orgdiv{Department}, \orgname{Organization}, \orgaddress{\street{Street}, \city{City}, \postcode{100190}, \state{State}, \country{Country}}}

% \affil[2]{\orgdiv{Department}, \orgname{Organization}, \orgaddress{\street{Street}, \city{City}, \postcode{10587}, \state{State}, \country{Country}}}

% \affil[3]{\orgdiv{Department}, \orgname{Organization}, \orgaddress{\street{Street}, \city{City}, \postcode{610101}, \state{State}, \country{Country}}}

\date{}
\maketitle
\begin{abstract}
  This paper studies the convergence of a spatial semidiscretization of a three-dimensional stochastic Allen-Cahn equation with multiplicative noise. For non-smooth initial data, the regularity of the mild solution is investigated, and an error estimate is derived within the spatial \( L^2 \)-norm setting.
  In the case of smooth initial data, two error estimates are established within the framework of general spatial \(L^q\)-norms.
\end{abstract}


\section{Introduction}
Let $ \mathcal{O} $ be a bounded, convex domain in $ \mathbb{R}^3 $ with a smooth boundary $ \partial\mathcal{O} $, and let $ \Delta $ denote the Laplace operator equipped with homogeneous Dirichlet boundary conditions. We consider the stochastic Allen–Cahn equation given by
\begin{subequations}
\label{eq:model}
\begin{numcases}{}
\mathrm{d}y(t) = \left( \Delta y + y - y^3 \right)(t) \, \mathrm{d}t + F(y(t)) \, \mathrm{d}W_H(t), \quad 0 \leqslant t \leqslant T, \\
y(0) = y_0,
\end{numcases}
\end{subequations}
where $ y_0 $ is the given initial condition, $ W_H $ is an $ H $-cylindrical Brownian motion, and the operator $ F $ will be specified later.

The stochastic Allen-Cahn equations are widely used to model
phase transitions in materials science;
see e.g.,~\cite{Berglund2019,Funaki1995,Kohn2007,Ondrejat2023,Roger2013}.
They are also an important class of nonlinear stochastic partial
differential equations from a mathematical perspective. Recently,
researchers in the field of numerical analysis have shown
considerable interest in this topic; see~\cite{Becker2023,Becker2019,
Brehier2019IMA,Brehier2020,Feng2017SIAM,Jentzen2020IMA,LiuQiao2019,
LiuQiao2021,Prohl2018,Ondrejat2023,QiWang2019,Wangxiaojie2020} and
the references therein. We briefly summarize some closely related
works of numerical analysis of three-dimensional stochastic
Allen-Cahn equations as follows.
Feng et al.~\cite{Feng2017SIAM} analyzed finite element approximations
of a stochastic Allen-Cahn equation with gradient-type multiplicative noise.
Majee and Prohl~\cite{Prohl2018} derived strong convergence rates for
a stochastic Allen-Cahn equation driven by a multiplicative Wiener process.
Br\'ehier et al.~\cite{Brehier2019IMA} analyzed an explicit temporal
splitting numerical scheme for a stochastic Allen-Cahn equation driven
by an additive Wiener process.
Jentzen and Pusnik~\cite{Jentzen2020IMA} proposed a numerical method for
approximating a class of semilinear stochastic evolution equations
with non-globally Lipschitz continuous nonlinearities.
Qi and Wang~\cite{QiWang2019} derived strong approximation errors for a
stochastic Allen-Cahn equation driven by an additive Wiener process.
Liu and Qiao~\cite{LiuQiao2021} established a general theory of the numerical
analysis for the stochastic Allen-Cahn equations with multiplicative noise.

The motivation for this work stems from two important gaps in the existing literature:
\begin{itemize}
  \item To the best of our knowledge, the numerical analysis of the three-dimensional stochastic Allen–Cahn equation
  with non-smooth initial data remains underdeveloped, despite its practical significance.
  Existing studies often impose restrictive regularity assumptions on the initial data.
  For instance, the analysis in \cite{LiuQiao2021} requires that the initial condition $ y_0 $ belongs to
  the space $ \dot{H}^{\alpha,2} $ for some $ \alpha \in [1,2] $, while \cite{QiWang2019} assumes $ \dot{H}^{1,2} $-regularity;
  the precise definition of these spaces is given in \cref{sec:pre}.
  The treatment of non-smooth initial data introduces substantial theoretical and numerical challenges.
  On the one hand, there is a lack of rigorous results for the three-dimensional stochastic Allen–Cahn
  equation with multiplicative noise and rough initial conditions (see, e.g., \cite{LiuWei2013, LiuWei2010}).
  On the other hand, the presence of singularities due to non-smooth initial data significantly complicates the numerical analysis.

  \item Current numerical studies on the three-dimensional stochastic Allen–Cahn equation with multiplicative noise predominantly focus on error estimates in the $L^2$-norm in space. However, the analysis of numerical approximations
  in more general spatial $L^q$-norms (for $q \geqslant 2$) remains largely unaddressed,
  which limits the scope and applicability of the theoretical results.
\end{itemize}

The main contributions of this work are summarized as follows:
\begin{itemize}
  \item For the model problem \cref{eq:model} with non-smooth initial data $ y_0 \in L^\infty(\mathcal{O}) $, we establish essential regularity estimates within the general framework of spatial $ L^q $-norms. These estimates are crucial for the subsequent numerical analysis. Specifically, we show that the mild solution satisfies
    \[
      y \in L_{\mathbb{F}}^p(\Omega; C([0,T]; L^q(\mathcal{O}))) \quad \text{for all } p, q \in (2, \infty).
    \]
    Our analysis combines key tools from both stochastic and deterministic PDE theory, including: stochastic integration and Itô's formula in UMD Banach spaces \cite{Weis_Ito_2008}; stochastic maximal $ L^p $-regularity estimates \cite{Neerven2012b, Neerven2012}; the Lyapunov framework for stochastic partial differential equations \cite{LiuWei2013, LiuWei2010}; and maximal $ L^p $-regularity estimates for deterministic evolution equations \cite{Weis2001}.
  
  \item To address the singularity induced by the non-smooth initial data, we employ a decomposition technique that separates the error into a singular and a more regular component. This allows us to derive the error estimate
    \[
      \norm{(y - y_h)(t)}_{L^p(\Omega; L^2(\mathcal{O}))} \leqslant c \big( h^{2/p - \epsilon} + h^{2/p} t^{-1/p} \big), \quad \forall t \in (0,T], \, p \in (2,\infty),
    \]
    where $ y_h $ is the strong solution of the spatial semidiscretization, $ \epsilon > 0 $ can be arbitrarily small, and $ c > 0 $ is a constant independent of $ t $ and the spatial mesh size $ h $. This result can be viewed as a stochastic counterpart of the classical deterministic error estimate given in \cite[Theorem~3.2]{Thomee2006}.
  
  \item For smooth initial data $ y_0 \in W_0^{1,\infty}(\mathcal{O}) \cap W^{2,\infty}(\mathcal{O}) $, we establish the optimal error estimate
    \begin{align*}
      \norm{y - y_h}_{L^p(\Omega \times (0,T); L^q(\mathcal{O}))} \leqslant c h^2,
    \end{align*}
    and the nearly optimal pathwise uniform error estimate
    \begin{align*}
      \norm{y - y_h}_{L^p(\Omega; C([0,T]; L^\infty(\mathcal{O}))} \leqslant c h^{2 - \epsilon} \quad \forall \epsilon > 0,
    \end{align*}
    where $ p, q \in (2,\infty) $ and $ c > 0 $ is a generic constant independent of the spatial mesh size $ h $. The analysis combines classical finite element theory \cite{Bakaev2002,Thomee2006} with discrete maximal $ L^p $-regularity \cite{Geissert2006} and discrete stochastic maximal $ L^p $-regularity \cite{LiLpSpatail2023}.
\end{itemize}

To the best of our knowledge, this work provides the first result on the spatial convergence of numerical approximations for the three-dimensional stochastic Allen–Cahn equation with multiplicative noise, established within the general framework of spatial $ L^q $-norms.
Although we assume deterministic initial data, the analysis can be naturally extended to the case of random initial conditions.
 The main contribution of this study lies in developing a rigorous and general numerical framework for spatial semidiscretizations of nonlinear stochastic parabolic equations under the framework of general spatial $ L^q $-norms, thereby distinguishing it from previous works.

% To establish second-order spatial convergence, prior numerical analyses typically impose stringent regularity assumptions
% on diffusion coefficients. By contrast, this work requires significantly milder regularity conditions on the diffusion
% coefficient \( F \) (Section \ref{sec:pre}). Although our analysis assumes deterministic initial data,
% extension to random initial data is readily accomplished. The principal contribution of this study consists
% in developing a rigorous numerical framework for nonlinear stochastic parabolic equations within the broad
% setting of spatial \( L^q \)-norms, distinguishing it from existing works.

% The remainder of this paper is organized as follows.
% Section \ref{sec:pre} introduces the notational conventions, definition of the operator $F$,
% the mild solution to model problem \ref{eq:model}, and the spatial semidiscretization.
% Section \ref{sec:main} presents the main results of this work, including the regularity properties of the
% mild solution to model problem \ref{eq:model} with non-smooth initial data,
% and establishes the convergence of the spatial semidiscretization for both non-smooth and smooth initial conditions.
% The rigorous proofs of these results are provided in Section \ref{sec:proof}.
% Finally, Section \ref{sec:conclusion} concludes the paper by summarizing the key findings of this study.

The remainder of this paper is structured as follows.
Section \ref{sec:pre} introduces notation, defines the operator $F$,
presents the mild solution to the model problem \cref{eq:model}, and describes its spatial semidiscretization.
Section \ref{sec:main} states the main results of this work,
analyzing the regularity properties of the mild solution under non-smooth initial data,
and establishes the convergence of the spatial semidiscretization for both non-smooth and smooth initial conditions.
Section \ref{sec:proof} provides the rigorous proofs of these results.
Finally, Section \ref{sec:conclusion} concludes the paper with a summary of the key findings.

\section{Preliminaries}
\label{sec:pre}
\medskip\noindent\textbf{Notational Conventions.}
For Banach spaces $ E_1 $ and $ E_2 $, we denote by $ \mathcal{L}(E_1, E_2) $ the space of all bounded linear operators from $ E_1 $ to $ E_2 $; in the case $ E_1 = E_2 $, we write $ \mathcal{L}(E_1) $ for brevity. The identity operator on any given space is denoted by $ I $.
For $ \theta \in (0,1) $ and $ p \in (1,\infty) $, the notation $ (\cdot, \cdot)_{\theta,p} $ refers
to the real interpolation space defined via the $ K $-method, as presented in Chapter 1 of \cite{Lunardi2018}.

Let $\mathcal{O} \subset \mathbb{R}^3$ denote a bounded convex domain with $\mathcal{C}^2$-regular boundary $\partial\mathcal{O}$, as defined in \cite[p.~39]{Yagi2010}. For $q \in [1,\infty]$, we adopt the abbreviation $L^q := L^q(\mathcal{O})$ for Lebesgue spaces and denote by $W_0^{1,q}(\mathcal{O})$ and $W^{2,q}(\mathcal{O})$ the standard Sobolev spaces, following the conventions established in \cite[Chapter~7]{Gilbarg2001}. The Laplace operator $\Delta$ is equipped with homogeneous Dirichlet boundary conditions.
For $\alpha \geqslant 0$ and $q \in (1,\infty)$, define $\dot{H}^{\alpha,q}$ as the domain of the operator $(-\Delta)^{\alpha/2}$ in $L^q$, endowed with the natural norm:
\[
\|v\|_{\dot{H}^{\alpha,q}} := \left\|(-\Delta)^{\alpha/2}v\right\|_{L^q}, \quad v \in \dot{H}^{\alpha,q}.
\]
The dual space of $\dot{H}^{\alpha,q}$ is denoted by $\dot{H}^{-\alpha,q'}$, where $q'$ is the Hölder conjugate of $q$.
When $\alpha = 2$, $\dot{H}^{2,q}$ coincides with $W_0^{1,q}(\mathcal{O}) \cap W^{2,q}(\mathcal{O})$,
and its norm is equivalent to the $W^{2,q}(\mathcal O)$-norm (cf. \cite[Theorem~9.15]{Gilbarg2001}).
For $\alpha = 1$, $\dot{H}^{1,q}$ agrees with $W_0^{1,q}(\mathcal{O})$, with equivalent norms (see \cite[Theorem~16.14]{Yagi2010}).

Let $ (\Omega, \mathcal{F}, \mathbb{P}) $ be a complete probability space endowed with a filtration $ \mathbb{F} := (\mathcal{F}_t)_{t \geqslant 0} $ satisfying the usual conditions.
On this stochastic basis, we define a sequence $ (W_n)_{n \in \mathbb{N}} $ of independent real-valued $ \mathbb{F} $-adapted Brownian motions. The associated expectation operator is denoted by $ \mathbb{E} $.
Let $ H $ be a separable Hilbert space with inner product $ (\cdot,\cdot)_H $ and an orthonormal basis $ (h_n)_{n \in \mathbb{N}} $.
The $ H $-cylindrical Brownian motion $ W_H $ is defined such that, for each $ t \in \mathbb{R}_{+} $, $ W_H(t) \in \mathcal{L}(H, L^2(\Omega)) $ is given by
$$
  W_H(t)x := \sum_{n \in \mathbb{N}} W_n(t) (x, h_n)_H, \quad \forall x \in H.
$$
We refer the reader to \cite{Neerven2007} for the stochastic integration with respect to $ W_H $.

Let $ p \in (1, \infty) $ and let $ E $ be a Banach space. Denote by $ L^p_{\mathbb{F}}(\Omega \times (0, T); E) $ the closure in the Bochner space $ L^p(\Omega \times (0, T); E) $ of the linear span of the set  
$$
\big\{ \mathbbm{1}_{(s,t] \times F} \otimes e : e \in E,\, 0 \leqslant s < t \leqslant T,\, F \in \mathcal{F}_s \big\},
$$
where $ \mathbbm{1}_{(s,t] \times F} $ denotes the indicator function of the set $ (s, t] \times F $.  
Let $ L^0_{\mathbb{F}}(\Omega; C([0, T]; E)) $ denote the space of all continuous, $ \mathbb{F} $-adapted, $ E $-valued processes.
Two processes $ g_1 $ and $ g_2 $ in $ L_{\mathbb F}^0(\Omega;C([0,T];E)) $ are considered identical if
$ g_1 = g_2 $ in $ C([0,T];E) $ holds $ \mathbb P $-almost surely.
We denote by $ L^p_{\mathbb{F}}(\Omega; C([0, T]; E)) $ the subspace of all processes $ g \in L^0_{\mathbb{F}}(\Omega; C([0, T]; E)) $ satisfying  
$$
\mathbb{E} \big[ \|g\|_{C([0, T]; E)}^p \big] < \infty,
$$
where $ C([0, T]; E) $ is the space of continuous $ E $-valued functions on $ [0, T] $, equipped with the supremum norm.
The space $ \gamma(H, E) $ denotes the space of $ \gamma $-radonifying operators from the Hilbert space $ H $ to $ E $ (see, e.g., \cite[Chapter 9]{HytonenWeis2017}).
If $ E $ is itself a Hilbert space, then $ \gamma(H, E) $ is isometrically isomorphic to the space of
Hilbert–Schmidt operators from $ H $ to $ E $.

\medskip\noindent\textbf{Definition of the Operator \(F\).}
Let $(f_n)_{n \in \mathbb{N}}$ be a sequence of real-valued, continuously differentiable functions defined on $\overline{\mathcal{O}} \times \mathbb{R}$, satisfying the following conditions:
\begin{align}
  & \sum_{n\in\mathbb N} |f_n(x,0)|^2 = 0,  && \forall x \in \partial\mathcal O,  \label{eq:fn0} \\
  & \sum_{n\in\mathbb N} \big( |f_n(x,y)|^2 + |\nabla_x f_n(x,y)|^2 \big) \leqslant C_F(1 + |y|^2), && \forall x \in \mathcal O, \, \forall y \in \mathbb R, \label{eq:fn-fnx-l2} \\
  &\sum_{n\in\mathbb N}  |\nabla_yf_n(x,y)|^2 \leqslant C_F, &&  \forall x \in \mathcal O, \, \forall y \in \mathbb R, \label{eq:fny-l2}
\end{align}
where $C_F > 0$ is a constant independent of $x$ and $y$. Here and in the sequel, we slightly abuse notation by identifying each $f_n$ with the corresponding Nemytskii operator defined pointwise by
$$
f_n(v) := f_n(\cdot, v(\cdot)), \quad \text{for all } v \in L^2.
$$
Recall that $(h_n)_{n \in \mathbb{N}}$ denotes a fixed orthonormal basis of the separable Hilbert space $H$.
Given condition \cref{eq:fn-fnx-l2}, for each $v \in L^2$, define the operator $F(v) \in \gamma(H, L^2)$ by
\begin{equation}
  \label{eq:F-def}
  F(v) := \sum_{n\in\mathbb N} h_n \otimes f_n(v),
\end{equation}
where the elementary tensor $h_n \otimes f_n(v) \in \gamma(H, L^2)$ acts on $x \in H$ via
$$
(h_n \otimes f_n(v))(x) := (x, h_n)_H f_n(v).
$$
In this setting, the model equation \eqref{eq:model} takes the explicit form of the stochastic evolution equation:
$$
  \begin{cases}
    \mathrm{d}y(t) = \left(\Delta y + y - y^3\right)(t) \, \mathrm{d}t + \displaystyle\sum_{n\in\mathbb N}
     f_n(y(t)) \,\mathrm{d}W_n(t), \quad t \in [0, T], \\
    y(0) = y_0,
  \end{cases}
$$
where $ \Delta $ denotes the Laplace operator with homogeneous Dirichlet boundary conditions.

\medskip\noindent\textbf{Mild Solutions.}
We define the two convolution operators:
\begin{align}
  (S_0 g)(t) &:= \int_0^t e^{(t-s)\Delta} g(s) \,\mathrm{d}s, \quad 0 \leqslant t \leqslant T,
  \label{eq:S0-def} \\
  (S_1 g)(t) &:= \int_0^t e^{(t-s)\Delta} g(s) \,\mathrm{d}W_H(s), \quad 0 \leqslant t \leqslant T,
  \label{eq:S1-def}
\end{align}
whenever the integrals are well-defined.
Here, $ (e^{t\Delta})_{t \geqslant 0} $ denotes the analytic semigroup generated by the Laplace operator $ \Delta $ under homogeneous Dirichlet boundary conditions.

Assume that $ y_0: (\Omega,\mathcal F_0,\mathbb P) \to L^4 $ is strongly measurable.
A process $ y \in L_{\mathbb F}^0(\Omega;C([0,T];L^4)) $ is called a mild solution to equation \eqref{eq:model} if it satisfies, $\mathbb{P}$-almost surely for all $ t \in [0,T] $,
\begin{equation}
  \label{eq:y-mild-def}
  y(t) = G(t) + \big[S_0(y - y^3)\big](t) + \big[S_1 F(y)\big](t),
\end{equation}
where
\begin{equation}
  \label{eq:G-def}
  G(t) := e^{t\Delta} y_0, \quad t \in [0, T].
\end{equation}

\medskip\noindent\textbf{Spatial Semidiscretization.}
Let $\mathcal{K}_h$ be a conforming and quasi-uniform triangulation of the domain $\mathcal{O}$, consisting of three-dimensional simplices. We define the finite element space $X_h$ as
$$
  X_h := \Bigl\{ v_h \in C(\overline{\mathcal{O}}) \Bigm| v_h = 0 \text{ on } \overline{\mathcal{O}} \setminus \mathcal{O}_h, \text{ and } v_h \text{ is piecewise linear with respect to each } K \in \mathcal{K}_h \Bigr\},
$$
where $C(\overline{\mathcal{O}})$ denotes the space of continuous functions on $\overline{\mathcal{O}}$, and $\mathcal{O}_h$ is the closure of the union of all elements in $\mathcal{K}_h$. We assume that every boundary vertex of $\mathcal{O}_h$ lies on $\partial\mathcal{O}$. The spatial mesh size $h$ is defined as the maximum diameter over all elements in $\mathcal{K}_h$.
Let $P_h$ denote the $L^2$-orthogonal projection onto $X_h$, and define the discrete Laplace operator $\Delta_h: X_h \to X_h$ via the identity
$$
  \int_{\mathcal{O}} (\Delta_h u_h) v_h \, \mathrm{d}x =
  -\int_{\mathcal{O}} \nabla u_h \cdot \nabla v_h \, \mathrm{d}x,
  \quad \text{for all } u_h, v_h \in X_h.
$$

We consider the following spatial semidiscretization of the model equation:
\begin{subequations}
\label{eq:yh}
\begin{numcases}{}
\mathrm{d}y_h(t) = \left( \Delta_h y_h + y_h - P_h y_h^3 \right)(t) \,\mathrm{d}t + P_h F(y_h(t)) \,\mathrm{d}W_H(t), \quad t \in [0, T], \\
y_h(0) = P_h y_0.
\end{numcases}
\end{subequations}
Suppose that $y_0: (\Omega,\mathcal{F}_0,\mathbb{P}) \to L^2(\mathcal{O})$ is strongly measurable. A process $y_h \in L_{\mathbb{F}}^0(\Omega;C([0,T];X_h))$, where $X_h$ is equipped with the $L^2$-norm, is called a strong solution to \eqref{eq:yh} if, $\mathbb{P}$-almost surely, the identity
$$
y_h(t) = y_h(0) + \int_0^t \left( \Delta_h y_h + y_h - P_h y_h^3 \right)(s) \,\mathrm{d}s + \int_0^t P_h F(y_h(s)) \,\mathrm{d}W_H(s)
$$
holds for all $t \in [0,T]$.
By the standard stability estimate for $P_h$ (see, e.g., \cref{eq:Ph-stab}),
together with the growth and Lipschitz conditions on $F$ given in \cref{lem:F-interp}, the classical theory of stochastic differential equations (see, e.g., \cite[Theorem~3.21]{Pardoux2014}) ensures the existence and uniqueness of a strong solution $y_h$ to the spatial semidiscretization \eqref{eq:yh}.

\section{Main results}
\label{sec:main}

% It is well-established that \( \Delta \) has a unique extension as a bounded linear operator from
% \( \dot{H}^{\alpha, q} \) to \( \dot{H}^{\alpha - 2, q} \) for all \( \alpha \in [0, 2) \).

In the remainder of this paper, we denote by $ c $ a generic positive constant that is independent of the spatial mesh size $ h $. It may, however,
depend on the final time $ T $, the domain $ \mathcal{O} $, the constant $ C_F $ from \cref{eq:fn-fnx-l2,eq:fny-l2},
the indices of the relevant Lebesgue, Sobolev, and interpolation spaces, the initial data $ y_0 $,
the shape-regularity parameters of the triangulation $ \mathcal{K}_h $, and the exponents associated with $ h $,
unless otherwise specified. The value of $ c $ may differ at each occurrence.

\begin{theorem}
  \label{thm:y-regu}
  Assume that \( y_0 \in L^\infty \) is deterministic. Then equation \cref{eq:model} admits a unique
  mild solution \( y \), which satisfies that
  \begin{equation}
    \label{eq:y-lplq}
    y \in L_\mathbb{F}^p(\Omega \times (0, T); L^q),
    \qquad \forall p, q \in (2, \infty).
  \end{equation}
  Furthermore, for any \( p,q \in (2, \infty) \) and \( \epsilon > 0 \),
  the solution \( y \) possesses the following regularity properties:
  \begin{align}
     & y \in L_{\mathbb F}^p\big( \Omega; C([0, T]; L^q) \big), \label{eq:y-C}            \\
     & y \in L_{\mathbb F}^p\big( \Omega \times (0, T); \dot{H}^{2/p - \epsilon, q} \big), \label{eq:y-lpHq} \\
     & y - G \in L_{\mathbb F}^p\big( \Omega; C([0, T]; \dot{H}^{1 - \epsilon, q}) \big), \label{eq:y-G-C} \\
     & y - G \in L_{\mathbb F}^p\big( \Omega \times (0, T); \dot{H}^{1 + 2/p - \epsilon, q} \big), \label{eq:y-G-lplq}
  \end{align}
  where \( G \) is defined by \cref{eq:G-def}.
\end{theorem}

\begin{theorem}
  \label{thm:conv}
  Assume that \( y_0 \in L^\infty \) is deterministic. Let \( y \) be the
  mild solution of \cref{eq:model}, and let \( y_h \) be the strong solution to the
  spatial semidiscretization \cref{eq:yh}. Then, for any \( p \in (2, \infty) \),
  the following error estimate holds for all \( t \in (0, T] \):
  \begin{equation}
    \label{eq:conv}
    \norm{(y - y_h)(t)}_{L^p(\Omega; L^2)} \leqslant
    c \left( h^{2/p-\epsilon} + h^{2/p} t^{-1/p} \right), \quad \forall \epsilon > 0.
  \end{equation}
\end{theorem}

\begin{theorem}
  \label{thm:conv-smooth}
  Assume that \( y_0 \in W_0^{1, \infty}(\mathcal{O}) \cap W^{2, \infty}(\mathcal{O}) \)
  is deterministic. Let \( y \) be the mild solution of \cref{eq:model},
  and let \( y_h \) be the strong solution of \cref{eq:yh}.
  Then the following error estimates hold:
  \begin{align}
     & \norm{y - y_h}_{L^p(\Omega \times (0, T); L^q)} \leqslant c h^2,
     \quad \forall p,q \in (2,\infty), \label{eq:y-yh-LpLq} \\
     & \norm{y - y_h}_{L^p(\Omega; C([0, T]; L^\infty))} \leqslant c h^{2 - \epsilon},
     \quad \forall p \in (2,\infty), \, \forall \epsilon > 0. \label{eq:y-yh-LpLinfty}
  \end{align}
\end{theorem}

It is worth noting that the aforementioned theoretical results remain valid
under the assumption that the operator \( F \) possesses the growth and Lipschitz continuity
properties outlined in Lemma \ref{lem:F-interp}.

\begin{remark}
  Liu~\cite{LiuWei2013} proposed a theoretical framework for a broad spectrum of nonlinear stochastic partial differential equations.
  Applying this framework to the model problem \cref{eq:model} yields
  \[
    y \in L^p(\Omega; C([0, T]; \dot{H}^{1, 2})),
    \quad \forall p \in (2, \infty),
  \]
  provided that \( y_0 \in L^p(\Omega,\mathcal F_0,\mathbb P;\dot{H}^{1, 2}) \), where $ p \in [2,\infty) $.
  Furthermore, Liu and Qiao \cite{LiuQiao2021} have established certain regularity estimates in the context of
  the spatial \( L^2 \)-norm, and some regularity estimates with the spatial \( L^\infty \)-norm were derived
  by using Sobolev's embedding theorem.
  Regarding the stochastic Allen-Cahn equation with additive noise, Bréhier et al.~\cite[Lemma~3.1]{Brehier2019IMA}
  provided some a priori estimates involving general spatial \( L^q \) norms.
\end{remark}

\begin{remark}
  In the numerical analysis of the three-dimensional stochastic Allen-Cahn equation with
  multiplicative noise, the existing literature often imposes strict regularity conditions
  on the diffusion coefficients. For instance, the study in \cite{Prohl2018} demands that
  the diffusion coefficients possess bounded second-order derivatives. Our choice of diffusion
  coefficient \( F \) adheres to the criteria outlined in Assumption 2.2 of \cite{LiuQiao2021};
  however, it does not meet the more stringent requirements stated in Assumption~2.3 of
  the same reference. Investigating the numerical approximations of \cref{eq:model}
  under less restrictive conditions on \( F \), especially the loosening of
  condition \eqref{eq:fn0}, has important theoretical and practical implications.
\end{remark}

\section{Proofs}\label{sec:proof}
The objective of this section is to affirm the validity of
Theorems \ref{thm:y-regu}, \ref{thm:conv}, and \ref{thm:conv-smooth}.
The structure of this section is as follows:
\begin{itemize}
  \item Subsection~\ref{ssec:pre} introduces key lemmas that serve as foundational results.
  \item Subsection~\ref{ssec:regu} focuses on the proof of Theorem \ref{thm:y-regu},
    which establishes the regularity of the mild solution of \cref{eq:model}.
  \item Subsection~\ref{ssec:conv} is dedicated to the proofs of Theorems
    \ref{thm:conv} and \ref{thm:conv-smooth}, which concern the error estimates of
    the spatial semidiscretization.
\end{itemize}

\subsection{Preliminary Lemmas}\label{ssec:pre}
It is a well-established result that for any \( \alpha \in (0,2) \) and $ q \in (1,\infty) $,
the space \( \dot H^{\alpha,q} \) coincides with the complex interpolation space \( [L^q, \dot H^{2,q}]_{\alpha/2} \),
equipped with equivalent norms. This result is documented, for example, in Chapter 16 of \cite{Yagi2010}.
The combination of Propositions 1.3 and 1.4 and Theorem 4.17 from \cite{Lunardi2018} and Theorem C.4.1 in \cite{HytonenWeis2016} yields
the following continuous embeddings.
\begin{lemma}
\label{lem:real-complex}  
Let $ 0 \leqslant \alpha < \beta \leqslant 2 $ and $ p, q \in (1, \infty) $. Then the following continuous embeddings hold:
\begin{itemize}
    \item[(i)] If $ \gamma > (1 - \theta)\alpha + \theta\beta $ for some $ \theta \in (0,1) $, then $ \dot{H}^{\gamma,q} $ is continuously embedded in the real interpolation space $ (\dot{H}^{\alpha,q}, \dot{H}^{\beta,q})_{\theta,p} $.

    \item[(ii)] If $ \gamma < (1 - \theta)\alpha + \theta\beta $ for some $ \theta \in (0,1) $, then the real interpolation space $ (\dot{H}^{\alpha,q}, \dot{H}^{\beta,q})_{\theta,p} $ is continuously embedded in $ \dot{H}^{\gamma,q} $.
\end{itemize}
\end{lemma}

Theorem 2.12 and inequality (2.128) in \cite{Yagi2010} provide the following standard
estimate for the growth of the analytic semigroup generated by the Laplace operator with
homogeneous Dirichlet boundary conditions.
\begin{lemma}
\label{lem:etDelta}
Let \( \alpha \in [0,\infty) \) and \( q \in (1,\infty) \). For all \( t > 0 \), the following estimate holds:
\[
\|(-\Delta)^{\alpha} e^{t\Delta}\|_{\mathcal L(L^q)}
+ t\Big\| (-\Delta)^{\alpha} \frac{\mathrm{d}}{\mathrm{d}t}e^{t\Delta}\Big\|_{\mathcal L(L^q)}\leqslant ct^{-\alpha}.
\]
\end{lemma}

For the convolution operator \( S_0 \), defined by \cref{eq:S0-def},
we have the following celebrated maximal $ L^p $-regularity estimate;
see Theorem 4.2 in \cite{Weis2001} and Corollary 1.14 from \cite{Lunardi2018}.
\begin{lemma}
\label{lem:S0-regu}
Let \( p, q \in (1,\infty) \). Then for any \( g \in L^p(0,T;L^q) \), the following inequality holds:
\begin{align*}
\big\lVert \frac{\mathrm{d}}{\mathrm{d}t} S_0 g \big\rVert_{L^p(0,T;L^q)} +
\norm{S_0 g}_{L^p(0,T;\dot H^{2,q})} +
\norm{S_0 g}_{C([0,T];(L^q,\dot H^{2,q})_{1-1/p,p})}
\leqslant c \norm{g}_{L^p(0,T;L^q)}.
\end{align*}
\end{lemma}

For the stochastic convolution operator \( S_1 \), defined by \cref{eq:S1-def},
we present the following fundamental stochastic maximal $ L^p$-regularity estimate;
see Theorems 1.1 and 1.2 in \cite{Neerven2012}.
\begin{lemma}
\label{lem:S1-regu}
Let \( p \in (2,\infty) \) and \( q \in [2,\infty) \). Then for any \( g \in L_{\mathbb{F}}^p(\Omega \times (0,T); \gamma(H, L^q)) \), the following inequality holds:
\begin{align*}
  \norm{S_1 g}_{L^p(\Omega; C([0,T]; (L^q, \dot{H}^{2,q})_{1/2-1/p,p}))} +
  \norm{S_1 g}_{L^p(\Omega \times (0,T); \dot{H}^{1,q})}                                 
  \leqslant c \norm{g}_{L^p(\Omega \times (0,T); \gamma(H, L^q))}.
\end{align*}
\end{lemma}

The operator $ F $, defined by \cref{eq:F-def}, possesses
the following growth and Lipschitz continuity properties.
\begin{lemma}
  \label{lem:F-interp}
  Let \( q \in [2, \infty) \). The following inequalities are valid:
  \begin{align}
    & \norm{F(v)}_{\gamma(H, L^q)} \leqslant c \left( 1 + \norm{v}_{L^q} \right), && \text{for all } v \in L^q; \label{eq:F-Lq} \\
    & \norm{F(v)}_{\gamma(H, \dot H^{1,q})} \leqslant c \left( 1 + \norm{v}_{\dot H^{1,q}} \right), && \text{for all } v \in \dot H^{1,q}; \label{eq:F-H1q} \\
    & \norm{F(v)}_{\gamma(H, \dot H^{s_1,q})} \leqslant c \left( 1 + \norm{v}_{\dot H^{s_2,q}} \right), && \text{for all } v \in \dot H^{s_2,q}, \text{ where } 0 \leqslant s_1 < s_2 < 1; \label{eq:F-interp} \\
    & \norm{F(u) - F(v)}_{\gamma(H, L^q)} \leqslant c \norm{u - v}_{L^q}, && \text{for all } u, v \in L^q. \label{eq:F-lips}
  \end{align}
\end{lemma}
\begin{proof}
  We focus our demonstration on \cref{eq:F-interp}; the proofs of \cref{eq:F-Lq}, \cref{eq:F-H1q},
  and \cref{eq:F-lips} follow similar, yet more straightforward, lines of reasoning.
  For the definition of $ \gamma $-radonifying operators, we direct the reader to Section~9.1 of \cite{HytonenWeis2017}.
  Consider $ \Omega_{\gamma} $ as a probability space supporting a sequence $ (\gamma_n)_{n\in\mathbb{N}} $
  of independent standard Gaussian random variables. Let $ \mathbb{E}_\gamma $ denote the expectation
  operator with respect to the probability space $ \Omega_{\gamma} $. The subsequent argument is
  structured in two main steps.

 \medskip\noindent\textbf{Step 1.}
We proceed to demonstrate that for any \( \theta \in (0,1) \) and \( N \in \mathbb N_{>0} \), the following inequality holds:
\begin{equation} 
  \label{eq:AN}
  \norm{A_Nv}_{L^2(\Omega_\gamma; (L^q, \dot H^{1,q})_{\theta,q})}
  \leqslant c \big( 1 + \norm{v}_{(L^q, \dot H^{1,q})_{\theta,q}} \big),
  \quad \forall v \in (L^q, \dot H^{1,q})_{\theta,q},
\end{equation}
where
\begin{align*}
  A_N v := \sum_{n=0}^N \gamma_n f_n(v), \quad \forall v \in L^q.
\end{align*}
For any \( v \in \dot{H}^{1,q} \), we deduce that
\begin{align*}
  \norm{A_Nv}_{L^q(\Omega_\gamma; \dot H^{1,q})}
  & \leqslant
  c\bigg[
    \mathbb E_\gamma \int_\mathcal O \absB{
      \sum_{n = 0}^N \gamma_n \nabla f_n(v)
    }^q \, \mathrm{d}x
  \bigg]^{1/q} \\
  & \stackrel{\mathrm{(i)}}{\leqslant}
  c\bigg[
    \int_\mathcal O \Big(
      \mathbb E_\gamma \absB{ \sum_{n=0}^N \gamma_n  \nabla f_n(v)}^2
    \Big)^{q/2} \, \mathrm{d}x
  \bigg]^{1/q}\\
  & = c\bigg[
    \int_\mathcal O \Big(
      \sum_{n=0}^N \abs{\nabla f_n(v)}^2
    \Big)^{q/2} \, \mathrm{d}x
  \bigg]^{1/q} \\
  & \stackrel{\mathrm{(ii)}}{\leqslant}
  c\bigg[
    \int_\mathcal O \Big(
      1 + \abs{v}^2 + \abs{\nabla v}^2
    \Big)^{q/2} \, \mathrm{d}x
  \bigg]^{1/q}, 
\end{align*}
where step \((\mathrm{i})\) uses the famous Kahane-Khintchine inequality (see Theorem 6.2.6 from \cite{HytonenWeis2017}),
 and step \((\mathrm{ii})\) relies on conditions \cref{eq:fny-l2,eq:fn-fnx-l2}.
Using the standard inequality
\[
  \left( \int_{\mathcal{O}} |v|^q + |\nabla v|^q \, \mathrm{d}x \right)^{1/q} \leqslant c \norm{v}_{\dot{H}^{1,q}},
\]
we further deduce that
\begin{align*}
  \norm{A_Nv}_{L^q(\Omega_\gamma; \dot H^{1,q})}
  \leqslant c(1 + \norm{v}_{\dot H^{1,q}}).
\end{align*}
Analogously, employing the condition \cref{eq:fn-fnx-l2} yields
\[
  \norm{A_Nv}_{L^q(\Omega_\gamma; L^q)} \leqslant c(1 + \norm{v}_{L^q}), \quad \forall v \in L^q,
\]
and utilizing the condition \cref{eq:fny-l2} results in
\[
  \norm{A_Nu - A_Nv}_{L^q(\Omega_\gamma;L^q)} \leqslant c\norm{u-v}_{L^q}, \quad \forall u, v \in L^q.
\]
Thus, following the methodology outlined in Lemma 28.1 of \cite{Tartar2007}, one obtains
\[
  \norm{A_Nv}_{(L^q(\Omega_\gamma; L^q), L^q(\Omega_\gamma; \dot{H}^{1,q}))_{\theta,q}} \leqslant c(1 + \norm{v}_{(L^q, \dot{H}^{1,q})_{\theta,q}}), \quad \forall v \in (L^q,\dot{H}^{1,q})_{\theta,q}.
\]
Combining this result with the continuous embedding property
\[
  \left( L^q(\Omega_\gamma;L^q), L^q(\Omega_\gamma; \dot{H}^{1,q}) \right)_{\theta,q} \hookrightarrow L^q\left(\Omega_\gamma; (L^q, \dot{H}^{1,q})_{\theta,q} \right),
\]
which can be easily verified by definition, yields
\[
  \norm{A_Nv}_{L^q(\Omega_\gamma; (L^q, \dot{H}^{1,q})_{\theta,q})} \leqslant c \big( 1 + \norm{v}_{(L^q, \dot{H}^{1,q})_{\theta,q}} \big), \quad \forall v \in (L^q, \dot{H}^{1,q})_{\theta,q}.
\]
Applying the Kahane-Khintchine inequality once again, we arrive at the assertion in \cref{eq:AN}.

\medskip\noindent\textbf{Step 2.}
For any \( \theta \in (0,1) \) and \( v \in (L^q, \dot{H}^{1,q})_{\theta,q} \),
we consider the equality
\[
\norm{F(v)}_{\gamma(H,(L^q,\dot H^{1,q})_{\theta,q})}
= \Big\lVert\lim_{N \to \infty} A_Nv\Big\rVert_{L^2(\Omega_\gamma; (L^q,\dot H^{1,q})_{\theta,q})},
\]
which follows directly from the definition of \(F\) as stated in (\ref{eq:F-def}) and is
justified by Theorem 9.1.19 in \cite{HytonenWeis2017}. Using the inequality (\ref{eq:AN}),
we apply Fatou's Lemma to deduce that
\[
\norm{F(v)}_{\gamma(H, (L^q, \dot H^{1,q})_{\theta,q})} \leqslant c(1 + \norm{v}_{(L^q, \dot H^{1,q})_{\theta,q}}).
\]
Consequently, the desired inequality (\ref{eq:F-interp}) is established through
the application of Lemma \ref{lem:real-complex}.
This concludes the proof.
\end{proof}

\begin{remark}
  As established by Theorem 4.36 in \cite{Lunardi2018}, the interpolation space \( (L^2, \dot H^{1,2})_{s,2} \) coincides with \( \dot H^{s,2} \), sharing equivalent norms for all \( s \in (0,1) \). Consequently, we refine the inequality \cref{eq:F-interp} to the following form:
  \[
    \| F(v) \|_{\gamma(H, \dot H^{s,2})} \leqslant c(1 + \| v \|_{\dot H^{s,2}}),
    \quad\text{for all $ v \in \dot H^{s,2} $, where $s \in (0,1)$}.
  \]
\end{remark}

It is well known that the space $ L^2 $ admits
an orthonormal basis $(\phi_n)_{n \in \mathbb{N}}$, consisting of
eigenvectors of the negative Laplace operator $ -\Delta $ with homogeneous
Dirichlet boundary conditions. This basis $(\phi_n)_{n \in \mathbb{N}}$ forms a
complete orthogonal system in the space $ \dot H^{s,2} $ for all $ s \in (0,\infty) $.

\begin{lemma}
  For each $ n \in \mathbb N $ and $ y_0^{(n)} \in \mathrm{span}\{\phi_0,\ldots,\phi_n\} $,
  which is deterministic, the stochastic Allen-Cahn equation
  \begin{subequations}
    \label{eq:yn}
    \begin{numcases}{}
      \mathrm{d}y^{(n)}(t) = \left[ \Delta y^{(n)} + y^{(n)} - (y^{(n)})^3 \right](t) \, \mathrm{d}t
      + F(y^{(n)}(t)) \, \mathrm{d}W_H(t), \quad 0 \leqslant t \leqslant T, \\
      y^{(n)}(0) = y_0^{(n)}
    \end{numcases}
  \end{subequations}
  has a unique mild solution. Moreover, for all $ p,q \in (2,\infty) $, the following inequality holds:
  \begin{equation}
    \norm{y^{(n)}}_{L_\mathbb F^p(\Omega\times(0,T);L^q)} \leqslant
    c (1 + \norm{y_0^{(n)}}_{L^q}).
    \label{eq:yn-esti1}
  \end{equation}
\end{lemma}
\begin{proof}
  We split the proof into the following two steps.

  \medskip\noindent\textbf{Step 1.} Define \( G^{(n)}(t) := e^{t\Delta}y_0^{(n)} \) for \( t \in [0,T] \).
  Given that \( y_0^{(n)} \in \dot H^{2,q} \) for every \( q \in (2,\infty) \),
  it follows from Lemma \ref{lem:etDelta} with \( \alpha=\beta=2 \) that
  \begin{equation}
    \label{eq:Gn-regu}
    G^{(n)} \in L^p(0,T;\dot H^{2,q}), \quad \forall p,q \in (2,\infty).
  \end{equation}
  In view of the estimates for $ F $ in \cref{eq:F-H1q,eq:F-lips} and the fact that
  \( y_0^{(n)} \in \dot H^{1,2} \) is deterministic,
  we use \cite[Theorem~1.1]{LiuWei2013} (see also, \cite[Lemma~2.3 and Theorem~3.1]{LiuQiao2021}) to
  deduce that \cref{eq:yn} has a unique mild solution \( y^{(n)} \) and
  \begin{equation}
    \label{eq:miss0}
    y^{(n)} \in L_\mathbb{F}^p( \Omega\times(0,T);\dot H^{1,2} ),
    \quad \forall p \in (2,\infty).
  \end{equation}
  The continuous embedding of \( \dot H^{1,2} \) in \( L^6 \) then leads to
  \[
    y^{(n)} \in L_\mathbb{F}^p(\Omega\times(0,T);L^6), 
    \quad \forall p \in (2,\infty),
  \]
  which in turn implies that
  \( y^{(n)} - (y^{(n)})^3 \in L_\mathbb{F}^p(\Omega\times(0,T);L^2) \)
  for all \( p \in (2,\infty) \).
  Applying Lemma \ref{lem:S0-regu}, we obtain
  \begin{equation}
    \label{eq:miss1}
    S_0(y^{(n)}-(y^{(n)})^3) \in L_\mathbb{F}^p(\Omega\times(0,T); \dot H^{2,2}),
    \quad \forall p \in (2,\infty).
  \end{equation}
  Since \eqref{eq:F-H1q} and \eqref{eq:miss0} ensure that
  \( F(y^{(n)}) \in L_\mathbb{F}^p(\Omega\times(0,T);\gamma(H,\dot H^{1,2})) \)
  for every \( p \in (2,\infty) \), Lemma \ref{lem:S1-regu} allows us to conclude that
  \begin{equation}
    \label{eq:miss2}
    S_1F(y^{(n)}) \in L_\mathbb{F}^p(\Omega\times(0,T); \dot H^{2,2}),
    \quad \forall p \in (2,\infty).
  \end{equation}
  Considering the mild solution representation
  \[
    y^{(n)} = G^{(n)} + S_0(y^{(n)}-(y^{(n)})^3) + S_1F(y^{(n)}),
  \]
  and combining \eqref{eq:Gn-regu}, \eqref{eq:miss1}, and \eqref{eq:miss2}, we deduce that
  \[
    y^{(n)} \in L_\mathbb{F}^p(\Omega\times(0,T);\dot H^{2,2}),
    \quad \forall p \in (2,\infty).
  \]
  The continuous embeddings of \( \dot H^{2,2} \) into \( \dot H^{1,6} \) and \( \dot H^{2,2} \) into
  \( L^q \) for all \( q \in (2,\infty) \) yield
  \begin{align*}
& y^{(n)} \in L_\mathbb{F}^p(\Omega\times(0,T); L^q),
\quad \forall p,q \in (2,\infty), \\
& y^{(n)} \in L_\mathbb{F}^p(\Omega\times(0,T);\dot H^{1,6}),
\quad \forall p \in (2,\infty).
  \end{align*}
  Hence, \cref{eq:miss1,eq:miss2} can be refined to
  \begin{align}
    & S_0(y^{(n)}-(y^{(n)})^3) \in L_\mathbb{F}^p(\Omega\times(0,T);\dot H^{2,q}),
    \quad \forall p,q \in (2,\infty),
    \label{eq:790} \\
    & S_1F(y^{(n)}) \in L_\mathbb{F}^p(\Omega\times(0,T);\dot H^{2,6}),
    \quad \forall p \in (2,\infty). \notag
  \end{align}
  Combining these refined regularity results with \eqref{eq:Gn-regu}, we arrive at
  \[
    y^{(n)} \in L_\mathbb{F}^p(\Omega\times(0,T);\dot H^{2,6}),
    \quad \forall p \in (2,\infty),
  \]
  which, through the continuous embedding of \( \dot H^{2,6} \) into \( \dot H^{1,q} \) for all \( q \in (2,\infty) \), implies that
  \[
    y^{(n)} \in L_\mathbb{F}^p(\Omega\times(0,T);\dot H^{1,q}),
    \quad \forall p,q \in (2,\infty).
  \]
  This regularity result allows us to further improve \eqref{eq:miss2} to
  \[
    S_1F(y^{(n)}) \in L_\mathbb{F}^p(\Omega\times(0,T);\dot H^{2,q}),
    \quad \forall p,q \in (2,\infty),
  \]
  which, in conjunction with \eqref{eq:Gn-regu} and \eqref{eq:790}, yields
  \begin{equation}
    \label{eq:yn-regu}
    y^{(n)} \in L_\mathbb{F}^p(\Omega\times(0,T);\dot H^{2,q}),
    \quad \forall p,q \in (2,\infty).
  \end{equation}

  \medskip\noindent\textbf{Step 2.}
  Let $ p,q \in (2, \infty) $ be arbitrary but fixed.
  Considering the regularity assertion \cref{eq:yn-regu} and the fact, $ \mathbb P $-almost surely, 
  \[
    y^{(n)}(t) = y_0^{(n)} + \int_0^t \Delta y^{(n)}(s) + y^{(n)}(s) - (y^{(n)}(s))^3 \, \mathrm{d}s +
    \int_0^t F(y^{(n)}(s)) \, \mathrm{d}W_H(s), \quad t \in [0,T],
  \]
  we use It\^o's formula in UMD Banach spaces (see \cite[Theorem~2.4]{Weis_Ito_2008})
  to conclude that, for any $ t \in (0,T] $,
  \begin{align*} 
    & \mathbb E\Big[ \norm{y^{(n)}(t)}_{L^q}^p - \norm{y_0^{(n)}}_{L^q}^p \Big]  \\
    ={}
    & p\, \mathbb E \Biggl[
      \int_0^t \norm{y^{(n)}(s)}_{L^q}^{p-q} \int_\mathcal O
      \abs{y^{(n)}(s)}^{q-2} y^{(n)}(s) [\Delta y^{(n)}(s) + y^{(n)}(s) - (y^{(n)}(s))^3]
      \,\mathrm{d}x \, \mathrm{d}s
    \Biggr] \\
    & \quad {} + \frac{p(q-1)}2 \, \mathbb E\Biggl[
      \int_0^t \norm{y^{(n)}(s)}_{L^q}^{p-q} \, \int_\mathcal O
      \abs{y^{(n)}(s)}^{q-2} \sum_{k=0}^\infty \abs{f_k(y^{(n)}(s))}^2
      \,\mathrm{d}x \, \mathrm{d}s
    \Biggr] \\
    & \quad
    {} + \frac{p(p-q)}2 \, \mathbb E \Biggl\{
      \int_0^t \norm{y^{(n)}(s)}_{L^q}^{p-2q} \sum_{k=0}^\infty \bigg[
        \int_\mathcal O \abs{y^{(n)}(s)}^{q-2} y^{(n)}(s) f_k(y^{(n)}(s)) \, \mathrm{d}x
      \bigg]^2 \, \mathrm{d}s                                                            
    \Biggr\} \\
    =:{} &
    I_1 + I_2 + I_3.
  \end{align*}
  Considering the inequality
  \begin{align*}
    & \int_{\mathcal O} \abs{y^{(n)}(s)}^{q-2} y^{(n)}(s) [\Delta y^{(n)}(s) - (y^{(n)}(s))^3] \, \mathrm{d}x \\
    ={} & 
    -(q-1) \int_{\mathcal O} \abs{y^{(n)}(s)}^{q-2} \abs{\nabla y^{(n)}(s)}^2 \, \mathrm{d}x -
    \norm{y^{(n)}(s)}_{L^{q+2}}^{q+2} \\
    \leqslant{} & 0,
  \end{align*}
  we infer the bound for $I_1 $: 
  \[
    I_1 \leqslant p \, \mathbb E \bigg[ \int_0^t \norm{y^{(n)}(s)}_{L^q}^{p} \, \mathrm{d}s \bigg].
  \]
  For \( I_2 + I_3 \), using Hölder's inequality in conjunction with Young's inequality,
   and employing condition \cref{eq:fn-fnx-l2}, we derive the following bound:
  \[
    I_2 + I_3 \leqslant c \, \mathbb{E} \bigg[ \int_0^t 1 + \norm{y^{(n)}(s)}_{L^q}^p \, \mathrm{d}s \bigg].
  \]
  Combining the above bounds for $ I_1 $ and $ I_2 + I_3 $, we arrive at
  \begin{align*}
    \mathbb E \norm{y^{(n)}(t)}_{L^q}^p \leqslant
    \norm{y_0^{(n)}}_{L^q}^p +
    c \, \mathbb E \bigg[
      \int_0^t 1 + \norm{y^{(n)}(s)}_{L^q}^p \, \mathrm{d}s
    \bigg],
    \quad \forall t \in [0,T].
  \end{align*}
  An application of Gronwall’s inequality then yields
  $$
  \norm{y^{(n)}}_{C([0,T]; L^p(\Omega; L^q))} \leqslant c \big( 1 + \norm{y_0^{(n)}}_{L^q} \big),
  $$
  which implies the desired estimate \cref{eq:yn-esti1}. This completes the proof.
\end{proof}

  % Note that (see~\cite[Theorem 4.36]{Lunardi2018}) the space
  % $ (L^2, \dot H^{1,q})_{s,2} $ is identical to $ \dot H^{s,2} $
  % with equivalent norms, for each $ s \in (0,1) $. 
 %  Consequently, the estimate \cref{eq:F-interp} can be improved as follows:
  % \[
  % 	\norm{F(v)}_{\gamma(H, \dot H^{s,2})} \leqslant
  % 	c(1 + \norm{v}_{\dot H^{s,2}}), \quad \forall v \in \dot H^{s,2},
  % 	\, \forall s \in [0,1].
  % \]

\subsection{Proof of \texorpdfstring{\cref{thm:y-regu}}{}}
\label{ssec:regu}
We split the proof into the following three steps.

\medskip\noindent\textbf{Step 1.} We proceed to show that equation \cref{eq:model} has a unique mild solution $y$,
and that \cref{eq:y-lplq} is satisfied. Let $p \in (2, \infty)$ and $ q \in [4,\infty)$ be arbitrarily chosen.
For each $n \in \mathbb{N}$, denote by $y^{(n)}$ the mild solution of equation \cref{eq:yn},
where $y_0^{(n)}$ is defined as
\[
  y_0^{(n)} := \mathop{\mathrm{argmin}}_{w \in \mathrm{span}\{\phi_0,\ldots,\phi_n\}}
  \norm{y_0-w}_{L^{3q}}.
\]
For any $ m,n \in \mathbb N $, put $ e_{m,n} := y^{(m)} - y^{(n)} $. By definition, we infer
that
\[
  \mathrm{d}e_{m,n}(t) = (\Delta e_{m,n} + e_{m,n} + (y^{(n)})^3 - (y^{(m)})^3)(t)
  \, \mathrm{d}t + \big( F(y^{(m)}(t)) - F(y^{(n)}(t)) \big) \, \mathrm{d}W_H(t),
  \quad t \in [0,T].
\]
Observing that the integral over \( \mathcal O \) of 
\[
  \abs{e_{m,n}(t)}^{3q-2} e_{m,n}(t) [\Delta e_{m,n}(t) + (y^{(n)}(t))^3 - (y^{(m)}(t))^3] 
\]
is non-positive for all $ t \in [0,T] $,
and following the proof of \cref{eq:yn-esti1},
we can deduce that
\begin{equation}
  \label{eq:emn}
  \norm{e_{m,n}}_{L_\mathbb F^{3p}(\Omega\times(0,T);L^{3q})}
  \leqslant c \norm{y_0^{(m)}-y_0^{(n)}}_{L^{3q}}.
\end{equation}
On the other hand, from the fact that $ \dot H^{2,2} $ is dense
in $ L^{3q} $ and that $ (\phi_n)_{n \in \mathbb N} $
is a complete orthogonal system in $ \dot H^{2,2} $, we conclude that
$ (y_0^{(n)})_{n \in \mathbb N} $ is a Cauchy sequence in $ L^{3q} $ and converges to $ y_0 $ in $ L^{3q} $,
i.e.,
\begin{equation}
  \label{eq:y0n-y0}
  \lim_{n \to \infty} \norm{y_0^{(n)} - y_0}_{L^{3q}} = 0.
\end{equation}
Consequently, there exists a unique process $ y \in L_{\mathbb F}^{3p}(\Omega\times(0,T);L^{3q}) $
such that
\begin{equation}
  \label{eq:yn-y}
  \lim_{n \to \infty} y^{(n)} = y \quad \text{ in }
  L_{\mathbb F}^{3p}(\Omega\times(0,T);L^{3q}).
\end{equation}
For any $ n \in \mathbb N $, applying H\"older's inequality yields
\begin{align*}
              & \norm{(y^{(n)})^3 - y^3}_{L^p(\Omega\times(0,T);L^q)}    \\
  \leqslant{} & c\norm{y^{(n)}-y}_{L^{3p}(\Omega\times(0,T);L^{3q})}
  \big(
  \norm{y^{(n)}}_{L^{3p}(\Omega\times(0,T);L^{3q})}^2 +
  \norm{y}_{L^{3p}(\Omega\times(0,T);L^{3q})}^2
  \big).
\end{align*}
Therefore, using the convergence established in \cref{eq:yn-y},
we deduce that
\[
  \lim_{n \to \infty} (y^{(n)})^3 = y^3 \quad \text{ in }
  L_{\mathbb F}^{p}(\Omega\times(0,T);L^{q}),
\]
which consequently leads to
\[
  \lim_{n \to \infty} y^{(n)} - (y^{(n)})^3 = y-y^3 \quad \text{ in }
  L_{\mathbb F}^{p}(\Omega\times(0,T);L^{q}).
\]
Using \cref{lem:S0-regu}, this directly implies the following 
convergence of $ S_0(y^{(n)}-(y^{(n)})^3) $:
\begin{equation*}
  \lim_{n \to \infty} S_0(y^{(n)} - (y^{(n)})^3) =
  S_0( y-y^3 ) \quad \text{ in }
  L_{\mathbb F}^{p}(\Omega\times(0,T);L^{q}) 
  \cap L_{\mathbb F}^p(\Omega;C([0,T];L^q)).
\end{equation*}
Furthermore, invoking \cref{eq:F-lips,eq:yn-y} yields 
\[
  \lim_{n \to \infty} F(y^{(n)}) = F(y) \quad \text{ in }
  L_\mathbb F^{p}(\Omega\times(0,T);\gamma(H,L^{q})),
\]
and subsequently by \cref{lem:S1-regu} we infer the following convergence for
$ S_1F(y^{(n)}) $:
\begin{equation*}
  \lim_{n \to \infty} S_1F(y^{(n)}) = S_1F(y) \quad\text{ in }
  L_\mathbb F^p(\Omega\times(0,T);L^q) \cap L_{\mathbb F}^p(\Omega;C([0,T];L^q)).
\end{equation*}
Additionally, applying \cref{lem:etDelta} with $ \alpha=\beta = 0 $ and
using \cref{eq:y0n-y0}, we establish the following convergence for $ G^{(n)} $:
\begin{equation*}
  \lim_{n \to \infty} G^{(n)} = G \quad\text{ in }
  C([0,T];L^q),
\end{equation*}
where $ G $ is defined by \cref{eq:G-def}, and
$ G^{(n)}(t) := e^{t\Delta} y_0^{(n)} $ for each $ t \in [0,T] $.
Finally, combining the derived convergences of $ y^{(n)} $, $ G^{(n)} $,
$ S_0(y^{(n)}-(y^{(n)})^3) $, and $ S_1F(y^{(n)}) $, and recalling that
$ y^{(n)} = G^{(n)} + S_0(y^{(n)} - (y^{(n)})^3) + S_1 F(y^{(n)}) $,
we obtain by passing to the limit $ n \to \infty $ that
\[
  y = G + S_0(y-y^3) + S_1F(y)
  \quad\text{in } L_\mathbb F^p(\Omega\times(0,T);L^q) \cap L_{\mathbb F}^p(\Omega;C([0,T];L^q)).
\]
This confirms that $ y $ is a mild solution of \cref{eq:model},
and it is evident that \cref{eq:y-lplq} holds.

\medskip\noindent\textbf{Step 2.}
We now consider the uniqueness of the mild solution $ y $, which can be proved by a 
standard localization argument. Suppose $ \widetilde{y} \in L_{\mathbb{F}}^0(\Omega; C([0,T]; L^4)) $ is another mild solution of \cref{eq:model}, and define the difference $ e := y - \widetilde{y} $. Then $ e $ satisfies the stochastic evolution equation  
$$
\begin{cases}
\mathrm{d}e(t) = \left( \Delta e(t) + e(t) + \widetilde{y}^3(t) - y^3(t) \right) \mathrm{d}t \\ 
\qquad\qquad {} + \left( F(y(t)) - F(\widetilde{y}(t)) \right) \mathrm{d}W_H(t)
\quad\text{in $\dot H^{-1,2}$,} \quad t \in [0,T], \\
e(0) = 0.
\end{cases}
$$  
For $ m \geqslant 1 $, define the stopping time  
$$
\sigma_m := \inf \left\{ t \in [0,T] \colon \max\left( \|y(t)\|_{L^4}, \|\widetilde{y}(t)\|_{L^4} \right) \geqslant m \right\},
$$  
with $ \sigma_m = T $ if the set is empty. Set $ e_m(t) := e(t \wedge \sigma_m) $ for $ t \in [0,T] $. Then $ e_m $ satisfies the stopped equation  
$$
\begin{cases}
\mathrm{d}e_m(t) = \mathbbm{1}_{[0,\sigma_m]}(t) \left( \Delta e_m(t) + e_m(t) + \widetilde{y}^3(t \wedge \sigma_m) - y^3(t \wedge \sigma_m) \right) \mathrm{d}t \\ 
\qquad\qquad {} + \mathbbm{1}_{[0,\sigma_m]}(t) \left( F(y(t)) - F(\widetilde{y}(t)) \right) \mathrm{d}W_H(t)
\quad\text{in $\dot H^{-1,2}$},\quad t \in [0,T], \\ 
e_m(0) = 0.
\end{cases}
$$  
Applying Itô’s formula to $ \|e_m(t)\|_{L^2}^2 $ and taking expectation yields, for all $ t \in [0,T] $,  
$$
\begin{aligned}
\mathbb{E} \|e_m(t)\|_{L^2}^2 &= 2\mathbb{E} \bigg[ \mathbbm{1}_{[0,\sigma_m]}(s) \big\langle \Delta e_m(s) + e_m(s) + \widetilde{y}^3(s) - y^3(s), \,  e_m(s) \big\rangle \, \mathrm{d}s \bigg] \\ 
&\quad + \mathbb{E} \bigg[ \int_0^t \mathbbm{1}_{[0,\sigma_m]}(s) \| F(y(s)) - F(\widetilde{y}(s)) \|_{\gamma(H, L^2)}^2 \, \mathrm{d}s \bigg],
\end{aligned}
$$  
where $ \langle \cdot, \cdot \rangle $ denotes the duality pairing between $ \dot H^{-1,2} $ and $ \dot H^{1,2} $.  Since  
$$
\mathbbm{1}_{[0,\sigma_m]}(s) \big\langle \Delta e_m(s) + e_m(s) + \widetilde{y}^3(s) - y^3(s), \, e_m(s) \big\rangle \leqslant \mathbbm{1}_{[0,\sigma_m]}(s) \|e_m(s)\|_{L^2}^2,
$$  
it follows that  
$$
\mathbb{E} \|e_m(t)\|_{L^2}^2 \leqslant \mathbb{E} \bigg[ \int_0^t \mathbbm{1}_{[0,\sigma_m]}(s) \left( \|e_m(s)\|_{L^2}^2 + \| F(y(s)) - F(\widetilde{y}(s)) \|_{\gamma(H, L^2)}^2 \right) \mathrm{d}s \bigg].
$$  
Using the Lipschitz condition \cref{eq:F-lips} of $ F $, we deduce  
$$
\mathbb{E} \|e_m(t)\|_{L^2}^2 \leqslant c \mathbb{E} \int_0^t \|e_m(s)\|_{L^2}^2 \, \mathrm{d}s, \quad \forall t \in [0,T].
$$  
Gronwall's inequality then yields  
$$
\sup_{t \in [0,T]} \mathbb{E} \|e_m(t)\|_{L^2}^2 = 0, \quad \forall m \geqslant 1.
$$  
Since $ e_m \in L_{\mathbb{F}}^0(\Omega; C([0,T]; L^4)) $, this implies $ \|e_m\|_{C([0,T];L^4)} = 0 $ holds $\mathbb{P}$-almost surely. Consequently,  
$$
y = \widetilde{y} \quad \text{in } C([0,\sigma_m];L^4) \text{ holds } \mathbb{P}\text{-almost surely, } \forall m \geqslant 1.
$$  
As $ y, \widetilde{y} \in L_{\mathbb{F}}^0(\Omega; C([0,T]; L^4)) $, we have $ \lim_{m \to \infty} \mathbb{P}(\sigma_m = T) = 1 $.  
Thus, $ y = \widetilde{y} $ in $ C([0,T];L^4) $ holds $\mathbb{P}$-almost surely, proving uniqueness for the mild solution to \cref{eq:model}.

\medskip\noindent\textbf{Step 3.} 
Fix any $ p,q \in (2,\infty) $.
By the regularity assertion in \cref{eq:y-lplq} and the growth estimate \cref{eq:F-Lq} for $ F $,
we deduce that $ y-y^3 $ belongs to $ L_\mathbb F^p(\Omega\times(0,T); L^q) $
and $ F(y) $ is an element of $ L_\mathbb F^p(\Omega\times(0,T); \gamma(H,L^q)) $.
Hence, the application of \cref{lem:S0-regu}, in conjunction with \cref{lem:real-complex},
gives that
\begin{equation}
  \label{eq:S0y}
  S_0(y-y^3) \in L^p\big(\Omega; L^p(0,T;\dot H^{2,q})
  \cap C([0,T]; \dot H^{2-2/p-\epsilon,q}) \big),
  \quad \forall\epsilon > 0,
\end{equation}
while \cref{lem:S1-regu} implies
\(
  S_1F(y) \in L_\mathbb F^p(\Omega\times(0,T);\dot H^{1,q})
\).
Additionally, the initial condition $y_0 \in L^\infty $ ensures by Lemma \ref{lem:etDelta} that
\(
  G \in L^p(0,T;\dot H^{2/p-\epsilon,q})
\)
for all $ \epsilon > 0 $. Consequently, by the mild solution representation \cref{eq:y-mild-def},
we arrive at
\[
  y \in L_\mathbb F^p(\Omega\times(0,T); \dot H^{2/p-\epsilon,q}), 
  \quad \forall \epsilon > 0,
\]
which confirms the regularity assertion \cref{eq:y-lpHq}.
In conjunction with \cref{eq:F-interp}, we further infer that
\[
  F(y) \in L_\mathbb F^p\big(
  \Omega\times(0,T); \gamma(H,\dot H^{2/p-\epsilon,q})
  \big), \quad \forall \epsilon > 0,
\]
leading to the enhanced conclusion via Lemmas \ref{lem:real-complex} and \ref{lem:S1-regu} that
\begin{equation}
  \label{eq:S1Fy}
  S_1F(y) \in L^p\big(
  \Omega; L^p(0,T;\dot H^{1+2/p-\epsilon,q})
  \cap C([0,T]; \dot H^{1-\epsilon,q})
  \big), \quad \forall \epsilon > 0.
\end{equation}
Finally, combining \cref{eq:y-mild-def,eq:S0y,eq:S1Fy} establishes
the regularity assertions in \cref{eq:y-G-C,eq:y-G-lplq}.
Moreover, the regularity assertion in \cref{eq:y-C} follows directly from \cref{eq:y-G-C}
and the observation that $ G \in C([0,T];L^q) $,
which is justified by the initial condition $y_0 \in L^\infty $ and \cref{lem:etDelta}. 
This completes the proof of \cref{thm:y-regu}.

\hfill{$\blacksquare$}

% By \cref{lem:etDelta} we have
% \begin{align*}
%   \norm{(-\Delta)^{1/p-\epsilon}G}_{L^p(0,T;L^q)}
%   \leqslant \norm{v}_{L^\infty}.
% \end{align*}
% We obtain
% \begin{align*}
%   \norm{(-\Delta)^{1/2+1/p-\epsilon} S_1 F(y)}_{L^p(\Omega\times(0,T);L^q)}
%   \leqslant c(1 + \norm{v}_{L^\infty}).
% \end{align*}
% It follows that
% \[
%   \norm{(-\Delta)^{1/2+1/p-\epsilon}(y - G)}_{L^p(\Omega\times(0,T;L^q))}
%   \leqslant c(1+\norm{v}_{L^\infty}).
% \]

\subsection{Proof of \texorpdfstring{\cref{thm:conv,thm:conv-smooth}}{}}
\label{ssec:conv}
We begin by summarizing several well-known results. It is known from \cite{Douglas1975} that
\begin{equation}
\label{eq:Ph-stab}
\|P_h\|_{\mathcal{L}(L^q)} \text{ is uniformly bounded in } h \text{ for all } q \in [1,\infty].
\end{equation}  
This uniform stability of $ P_h $ implies the following approximation property:
\begin{equation}
\label{eq:Ph-conv}
\|I - P_h\|_{\mathcal{L}(\dot{H}^{\alpha,q}, L^q)} \leqslant c h^\alpha
\quad\text{for all } \alpha \in [0,2] \text{ and } q \in (1,\infty).
\end{equation}  
Furthermore, it is a standard result (see, e.g., \cite[Equations (12) and (13)]{Geissert2007}) that
\begin{align}
  \norm{I - \Delta_h^{-1}P_h\Delta}_{\mathcal{L}(\dot H^{2,q}, L^q)}
  & \leqslant c h^2, \quad \forall\, q \in (1,\infty), \label{eq:Ritz-conv-H2} \\
  \norm{I - \Delta_h^{-1}P_h\Delta}_{\mathcal{L}(\dot H^{1,q}, L^q)}
  & \leqslant c h, \quad \forall\, q \in [2,\infty). \label{eq:Ritz-conv-H1}
\end{align}
By applying complex interpolation between \eqref{eq:Ritz-conv-H2} and \eqref{eq:Ritz-conv-H1}, one obtains
\begin{equation}
  \label{eq:Ritz}
  \norm{I - \Delta_h^{-1}P_h\Delta}_{\mathcal{L}(\dot H^{\alpha,q}, L^q)}
  \leqslant c h^\alpha, \quad \forall\, \alpha \in [1,2), \quad \forall\, q \in [2,\infty).
\end{equation}
Combining this with \cref{eq:Ph-conv} via the triangle inequality yields
\begin{equation}
  \label{eq:Ph-DeltahPhDelta}
  \norm{P_h - \Delta_h^{-1}P_h\Delta}_{\mathcal{L}(\dot H^{\alpha,q}, L^q)}
  \leqslant c h^\alpha, \quad \forall \alpha \in [1,2], \, \forall q \in [2,\infty),
\end{equation}
which in turn implies the estimate
\begin{equation}
  \label{eq:Ritz-conv-Halpha}
  \norm{P_h\Delta^{-1} - \Delta_h^{-1}P_h}_{\mathcal{L}(\dot H^{\alpha-2,q}, L^q)}
  \leqslant c h^\alpha, \quad \forall \alpha \in [1,2], \, \forall q \in [2,\infty).
\end{equation}

\begin{remark}
  The projection operator $ P_h $ can be extended to act as a linear operator
  from $ \dot{H}^{-1,q} $ to $ X_h $ for any $ q \in (1, \infty) $. This extension is defined via the identity  
  $$
  \int_{\mathcal{O}} (P_h v) v_h \, \mathrm{d}x = \dual{v, v_h},
  $$  
which holds for all $ v \in \dot{H}^{-1,q} $ and $ v_h \in X_h $. Here, $ \dual{\cdot, \cdot} $ denotes the duality pairing between $ \dot{H}^{-1,q} $ and $ \dot{H}^{1,q'} $.
\end{remark}

Next, we introduce the discrete Sobolev space and recall some standard results.
Let $ q \in (1, \infty) $. For any $ \alpha \in \mathbb{R} $, define $ \dot{H}_h^{\alpha, q} $ as the space $ X_h $ equipped with the norm
$$
  \|v_h\|_{\dot{H}_h^{\alpha, q}} := \|(-\Delta_h)^{\alpha/2} v_h\|_{L^q}, \quad v_h \in X_h.
$$
According to Theorem 3.1 (ii) in \cite{LiLpSpatail2023}, the purely imaginary powers of $ -\Delta_h $ are bounded in $ \mathcal{L}(\dot{H}_h^{0, q}) $ uniformly in $ h $. As a consequence, an application of Theorem 4.17 from \cite{Lunardi2018} yields, for any $ \alpha \in (0, 2) $, the equivalence
$$
  c_0 \|v_h\|_{\dot{H}_h^{\alpha, q}} \leqslant \|v_h\|_{[\dot{H}_h^{0, q}, \dot{H}_h^{2, q}]_{\alpha/2}} \leqslant c_1 \|v_h\|_{\dot{H}_h^{\alpha, q}}, \quad v_h \in X_h,
$$
with constants $ c_0, c_1 > 0 $ independent of $ h $. Here, $ [\dot{H}_h^{0, q}, \dot{H}_h^{2, q}]_{\alpha/2} $ denotes the complex interpolation space between $ \dot{H}_h^{0, q} $ and $ \dot{H}_h^{2, q} $; see \cite[Chapter 2]{Lunardi2018} for details.
Given the standard inverse estimate
$$
  \|v_h\|_{\dot{H}_h^{2, q}} \leqslant c h^{-2} \|v_h\|_{L^q}, \quad v_h \in X_h,
$$
complex interpolation yields the following general inverse estimate:
\begin{equation}
  \label{eq:inverse}
  \|v_h\|_{\dot{H}_h^{\beta, q}} \leqslant c h^{\alpha - \beta} \|v_h\|_{\dot{H}_h^{\alpha, q}},
  \quad v_h \in X_h, \, \forall 0 \leqslant \alpha < \beta < \infty.
\end{equation}
Applying this inverse estimate to \cref{eq:Ritz-conv-Halpha}, we obtain the estimate
\begin{equation}
\label{eq:foo-conv2}
\|P_h\Delta^{-1} - \Delta_h^{-1}P_h\|_{\mathcal{L}(\dot{H}^{\alpha-2,q}, \dot{H}_h^{\beta,q})} \leqslant c h^{\alpha - \beta}
\quad \text{for all } \alpha \in [1,2],\, \beta \in [0,1],\, q \in [2,\infty).
\end{equation}
In analogy with \cref{lem:real-complex}, we also have the following $ h $-uniform embedding result.
\begin{lemma}
  \label{lem:real-complex-discrete}
  Let $ 0 \leqslant \beta < \gamma \leqslant 1 $ and $ p,q \in (1,\infty) $. Then
  $$
    \|v_h\|_{\dot H_h^{\beta,q}} \leqslant c\|v_h\|_{(\dot H_h^{0,q}, \dot H_h^{2,q})_{\gamma/2,p}}, \quad v_h \in X_h.
  $$
\end{lemma}
Additionally, we establish the following $ h $-uniform Sobolev-type embeddings into $ L^\infty $, whose proofs are given in \cref{sec:appendix}.
% \begin{lemma}
%   \label{lem:Hh-Linfty}
%   The following continuous embeddings hold:
%   \begin{enumerate}
%     \item[\rm{(i)}] If $ q \in (3/2, \infty) $, then
%       $ \|v_h\|_{L^\infty} \leqslant c \|v_h\|_{\dot H_h^{2,q}} $ for all $ v_h \in X_h $.
%     \item[\rm{(ii)}] If $ q \in (3, \infty) $, then
%       $ \|v_h\|_{L^\infty} \leqslant c \|v_h\|_{\dot H_h^{1,q}} $ for all $ v_h \in X_h $.
%     \item[\rm{(iii)}] If $ q \in (3,\infty) $ and $ \alpha \in (1,2] $, then
%       $ \|v_h\|_{L^\infty} \leqslant c \|v_h\|_{\dot H_h^{\alpha,q}} $ for all $ v_h \in X_h $.
%     \item[\rm{(iv)}] If $ \alpha \in (0,1) $ and $ q \in (3/\alpha,\infty) $, then
%       $ \|v_h\|_{L^\infty} \leqslant c \|v_h\|_{\dot H_h^{\alpha,q}} $ for all $ v_h \in X_h $.
%   \end{enumerate}
% \end{lemma}

\begin{lemma}
  \label{lem:Hh-Linfty}
  The following continuous embeddings hold:
  \begin{enumerate}
    \item[\rm{(i)}] If $ q \in (3/2, \infty) $, then
      $ \|v_h\|_{L^\infty} \leqslant c \|v_h\|_{\dot H_h^{2,q}} $ for all $ v_h \in X_h $.
    \item[\rm{(ii)}]
      If $ \alpha \in (0,2) $ and $ q \in (3/\alpha,\infty) \cap [2,\infty) $, then
      $ \|v_h\|_{L^\infty} \leqslant c\|v_h\|_{\dot H_h^{\alpha,q}} $ for all $ v_h \in X_h $.
  \end{enumerate}
\end{lemma}

Now, we introduce the discrete semigroup $ e^{t\Delta_h} $ and the maximal $ L^p $-regularity estimates
associated with this semigroup.
Let \( e^{t\Delta_h} \), for \( t \in [0,\infty) \), denote the analytic semigroup generated by the discrete Laplace operator \( \Delta_h \).
The following fundamental properties hold, with proofs provided in \Cref{sec:appendix}.
\begin{lemma}
  \label{lem:etDeltah}
  The following estimates hold:
  \begin{itemize}
    \item[\rm{(i)}]
      For any \( \alpha \in [0,1] \), \( q \in (1,\infty) \), and \( t >0 \),
      \begin{equation}
        \norm{(-\Delta_h)^{\alpha}e^{t\Delta_h}}_{\mathcal{L}(L^q)}
        \leqslant ct^{-\alpha}.
        \label{eq:etDeltah}
      \end{equation}
    \item[\rm{(ii)}] For any $ \alpha \in [0,2] $, $ q \in (1,\infty) $, and $ t > 0 $,
      \begin{align}
        \norm{e^{t\Delta_h}P_h - P_he^{t\Delta}}_{
          \mathcal{L}(\dot{H}^{\alpha,q},L^q)
        } + \norm{(I-P_h)e^{t\Delta}}_{
          \mathcal{L}(\dot{H}^{\alpha,q}, L^q)
        } \leqslant c \min\Big\{ h^\alpha, \frac{h^2}{t^{1-\alpha/2}} \Big\}.
        \label{eq:foo-conv4}
      \end{align}
  \end{itemize}
\end{lemma}

For a function $ g_h \in L^1(0,T;\dot{H}_h^{0,2}) $, the discrete convolution operator $ S_{0,h} $ is defined by
\begin{equation}
  \label{eq:S0h-def}
  (S_{0,h}g_h)(t) := \int_0^t e^{(t-s)\Delta_h} g_h(s) \, \mathrm{d}s, \quad t \in [0,T].
\end{equation}
For a stochastic process $ g_h \in L_\mathbb{F}^2(\Omega \times (0,T); \gamma(H, \dot{H}_h^{0,2})) $,
the discrete stochastic convolution operator $ S_{1,h} $ is given by
\begin{equation}
  \label{eq:S1h-def}
  (S_{1,h}g_h)(t) := \int_0^t e^{(t-s)\Delta_h} g_h(s) \, \mathrm{d}W_H(s), \quad t \in [0,T].
\end{equation}
These operators satisfy the following fundamental discrete maximal $ L^p $-regularity estimates.
\begin{lemma}
  \label{lem:S0h}
  Suppose $ p,q \in (1,\infty) $. For any $ g_h \in L^p(0,T;\dot H_h^{0,q}) $, the following
  discrete maximal $ L^p $-regularity estimate is valid:
  \[
    \norm{S_{0,h}g_h}_{C([0,T];(\dot{H}_h^{0,q},\dot H_h^{2,q})_{1-1/p,p}} +
    \norm{S_{0,h}g_h}_{L^p(0,T;\dot{H}_h^{2,q})} \leqslant
    c\norm{g_h}_{L^p(0,T;\dot{H}_h^{0,q})}.
  \]
\end{lemma}
\begin{lemma}
  \label{lem:S1h}
  Let $ p \in (2,\infty) $ and $ q \in [2,\infty) $.
  For any $ g_h \in L_\mathbb{F}^p(\Omega\times(0,T);\gamma(H,\dot H_h^{0,q})) $,
  the following discrete stochastic maximal $ L^p $-regularity estimate holds:
  \begin{equation}
    \begin{aligned}
      & \norm{S_{1,h}g_h}_{L^p(\Omega;C([0,T];(\dot H_h^{0,q}, \dot H_h^{2,q})_{1/2-1/p,p}))} +
      \norm{S_{1,h}g_h}_{L^p(\Omega\times(0,T);\dot H_h^{1,q})}                                    \\
      \leqslant{} &
      c \norm{g_h}_{L^p(\Omega\times(0,T);\gamma(H,\dot H_h^{0,q}))}.
    \end{aligned}
  \end{equation}
\end{lemma}
\begin{remark}
  A proof of Lemma~\ref{lem:S0h} can be found in \cite[Theorem~3.2]{Geissert2006}.
  Since the boundedness constant of the $ H^\infty $-calculus for $ -\Delta_h $ is 
  independent of the spatial mesh size $ h $, as shown in \cite[Theorem~3.1]{LiLpSpatail2023},
  Lemma~\ref{lem:S1h} follows directly from \cite[Theorem~3.5]{Neerven2012b};
  see \cite[Theorem~3.2]{LiLpSpatail2023} for further details.
\end{remark}

% \begin{proof}
%   For any \( t \in [0, T] \), from the identity \( S_{0,h}g_h(t) = \int_0^t e^{(t-s)\Delta_h}g_h(s) \, \mathrm{d}s \),
%   we deduce that 
%   \begin{align*}
%     \|S_{0,h}g_h(t)\|_{\dot{H}_h^{2-2/p-\epsilon, q}}
%     &\leqslant \int_0^t \left\|e^{(t-s)\Delta_h}g_h(s)\right\|_{\dot{H}_h^{2-2/p-\epsilon, q}} \, \mathrm{d}s \\
%     &\stackrel{\text{(i)}}{\leqslant} c\int_0^t (t-s)^{-1+\frac1p+\frac\epsilon2}
%     \|g_h(s)\|_{L^q} \, \mathrm{d}s \\
%     &\stackrel{\text{(ii)}}{\leqslant} c \Big(
%       \int_0^t (t-s)^{(-1+\frac1p+\frac\epsilon2)\frac{p}{p-1}} \, \mathrm{d}s
%     \Big)^{1-\frac1p} \|g_h\|_{L^p(0, T; L^q)},
%   \end{align*}
%   where step (i) uses the estimate \cref{eq:etDeltah} with \( \alpha = 1-\frac1p-\frac\epsilon2 \),
%   and step (ii) applies Hölder's inequality. Since the integral
%   \[
%     \int_0^t (t-s)^{\left(-1+\frac{1}{p}+\frac{\epsilon}{2}\right)\frac{p}{p-1}} \, \mathrm{d}s 
%   \]
%   is uniformly bounded with respect to \( t \in [0, T] \), the desired inequality follows, thereby completing the proof of the lemma.
% \end{proof}

We proceed to present some auxiliary estimates.

\begin{lemma}
  Suppose that $ y_0 \in L^\infty $ is deterministic. Consider the mild solution $ y $ to equation (\ref{eq:model}),
  the strong solution $ y_h $ to equation (\ref{eq:yh}), and the function $ G $ defined by equation
  (\ref{eq:G-def}). Additionally, let $ G_h(t) := e^{t\Delta_h} P_h y_0 $ for $ t \in [0,T] $.
  For notational simplicity, let $ z := y-G $ and $ \xi_h := S_{0,h}(\Delta_hP_h-P_h\Delta)z $.
  For any $ p,q \in (2,\infty) $ and $ \epsilon > 0 $,
  the following inequalities hold:
  \begin{align}
     & \norm{G-G_h}_{L^p(0,T;L^q)} + \norm{P_hG-G_h}_{L^p(0,T;L^q)}
     \leqslant ch^{\frac2p}, \label{eq:G-Gh-PhG} \\
     & \norm{z-P_hz}_{L^p(\Omega\times(0,T);L^q)}
     \leqslant ch^{1+\frac2p-\epsilon},  \label{eq:z-Phz} \\
     & \norm{\xi_h}_{L^p(\Omega\times(0,T);L^q)}
     \leqslant ch^{1+\frac2p-\epsilon}, \label{eq:Deltah-Delta} \\
     & \norm{\xi_h}_{L^p(\Omega;C([0,T];L^q))}
     \leqslant ch^{1-\epsilon}, \label{eq:Deltah-Delta-2}  \\
     & \mathbb E\bigg[
       \int_0^T \norm{(y-P_hy)(t)}_{L^4}^p \norm{y(t)}_{L^8}^{2p}
       \, \mathrm{d}t
     \bigg] \leqslant ch^{2-\epsilon}. \label{eq:y-Phy-y} 
  \end{align}
  Furthermore, for any $ \beta \in \Big(1, \frac{3p}{3p-4}\Big) $,
  the following inequalities hold:
  \begin{align}
     & \mathbb E\bigg[
       \int_0^T \norm{(P_hG-G_h)(t)}_{L^{2\beta'}}^p
       \big(
         \norm{y(t)}_{L^{2\beta}}^p +
         \norm{y_h(t)}_{L^{2\beta}}^p
       \big) \, \mathrm{d}t
     \bigg]
     \leqslant ch^{2 - \frac32\big(1-\frac1\beta\big)p}, 
     \label{eq:PhG-Gh-y}     \\
     & \mathbb E\bigg[
       \int_0^T \norm{\xi_h(t)}_{L^{2\beta'}}^p
       \big(
         \norm{y(t)}_{L^{2\beta}}^{p} +
         \norm{y_h(t)}_{L^{2\beta}}^p
       \big) \, \mathrm{d}t
     \bigg] \leqslant ch^{p + 2 -\frac32\big(1-\frac1\beta\big)p - \epsilon}.
     \label{eq:love}
  \end{align}
  Here, $ \beta' $ denotes the H\"older conjugate of $ \beta $.
\end{lemma}
\begin{proof}
  The inequality \cref{eq:G-Gh-PhG} is evident by \cref{eq:foo-conv4},
  and the inequality \cref{eq:z-Phz} follows by \cref{eq:Ph-conv} and
  the regularity assertion \cref{eq:y-G-lplq}.
  The inequality \cref{eq:Deltah-Delta} is derived as follows:
  \begin{align*}
    \norm{\xi_h}_{L^p(\Omega\times(0,T);L^q)} 
    &\leqslant
    c\norm{\Delta_h^{-1}(\Delta_hP_h-P_h\Delta)z}_{L^p(\Omega\times(0,T);\dot H_h^{0,q})}
    \quad\text{(by \cref{lem:S0h})}  \\
    &=
    c\norm{(P_h\Delta^{-1}-\Delta_h^{-1}P_h)\Delta z}_{L^p(\Omega\times(0,T);\dot H_h^{0,q})} \\
    &\leqslant
    ch^{1+2/p-\epsilon} \norm{\Delta z}_{L^p(\Omega\times(0,T);\dot H^{-1+2/p-\epsilon,q})}
    \quad\text{(by \cref{eq:foo-conv2})} \\
    &=
    ch^{1+2/p-\epsilon} \norm{z}_{L^p(\Omega\times(0,T);\dot H^{1+2/p-\epsilon,q})} \\
    &\leqslant
    ch^{1+2/p-\epsilon} \quad
    \text{(by \cref{eq:y-G-lplq})}.
  \end{align*}
  The inequality \cref{eq:Deltah-Delta-2} is derived in a similar manner:
  \begin{align*}
    & \norm{\xi_h}_{L^p(\Omega;C([0,T];L^q))} \\
    \leqslant{}
    & c\norm{(\Delta_hP_h-P_h\Delta)z}_{L^p(\Omega\times(0,T);\dot H_h^{-2+2/p+\epsilon/2,q})}
    \quad\text{(by \cref{lem:S0h,lem:real-complex-discrete})} \\
    ={}
    & c\norm{(P_h\Delta^{-1}-\Delta_h^{-1}P_h)\Delta z}_{L^p(\Omega\times(0,T);\dot H_h^{2/p+\epsilon/2,q})} \\
    \leqslant{}
    & ch^{1-\epsilon} \norm{\Delta z}_{L^p(\Omega\times(0,T); \dot H^{-1+2/p-\epsilon/2,q})} \quad
    \text{(by \cref{eq:foo-conv2})} \\
    ={}
    & ch^{1-\epsilon} \norm{z}_{L^p(\Omega\times(0,T); \dot H^{1+2/p-\epsilon/2,q})} \\
    \leqslant{}
    & ch^{1-\epsilon} \quad \text{(by \cref{eq:y-G-lplq})}.
  \end{align*}
For \cref{eq:y-Phy-y}, we apply H\"older's inequality with $ \alpha \in (1,2) $, yielding
\begin{align*}
  & \mathbb{E} \left[ \int_0^T \norm{(y - P_h y)(t)}_{L^4}^p \norm{y(t)}_{L^8}^{2p} \, \mathrm{d}t \right] \\
  \leqslant{} & \norm{y - P_h y}_{L^{\alpha p}(\Omega \times (0,T); L^4)}^p \norm{y}_{L^{\frac{2\alpha p}{\alpha - 1}}(\Omega \times (0,T); L^8)}^{2p} \\
  \leqslant{} & c \norm{y - P_h y}_{L^{\alpha p}(\Omega \times (0,T); L^4)}^p \quad \text{(by \cref{eq:y-lplq})} \\
  \leqslant{} & c h^{\frac{2}{\alpha}  - \frac{\epsilon}{2} } \norm{y}_{L^{\alpha p}(\Omega \times (0,T); \dot{H}^{\frac{2}{\alpha p} - \frac{\epsilon}{2p},4})}^p \quad \text{(by \cref{eq:Ph-conv})}.
\end{align*}
Choosing $ \alpha > 1 $ sufficiently close to 1 so that $ \frac{2}{\alpha} - \frac{\epsilon}{2} \geqslant 2 - \epsilon $, and using the regularity estimate \eqref{eq:y-lpHq}, we obtain \eqref{eq:y-Phy-y}.

Now, let us prove \cref{eq:PhG-Gh-y}. Fix any $ \beta \in (1,\frac{3p}{3p-4}) $.
By using \cref{eq:Ph-stab} and following the reasoning behind
\cite[Equation (2.2)]{LiuWei2010}, we establish that
\[
  \mathbb{E}\bigg[
    \int_0^T \norm{y_h(t)}_{L^2}^{\frac{4\beta}{3\beta-3}-2} \norm{y_h(t)}_{\dot{H}^{1,2}}^2 \, \mathrm{d}t
  \bigg]
\]
is uniformly bounded with respect to $h$. Using the continuous embedding of
$\dot{H}^{1,2}$ into $L^6$, this uniform bound extends to
\[
  \mathbb{E}\bigg[
    \int_0^T \norm{y_h(t)}_{L^2}^{\frac{4\beta}{3\beta-3}-2} \norm{y_h(t)}_{L^6}^2 \, \mathrm{d}t
  \bigg].
\]
Employing H\"older's inequality, we deduce for all $t \in [0,T]$ that
\[
  \norm{y_h(t)}_{L^{2\beta}}^{\frac{4\beta}{3\beta-3}} \leqslant
  \norm{y_h(t)}_{L^2}^{\frac{4\beta}{3\beta-3}-2} \norm{y_h(t)}_{L^6}^2,
\]
which in turn confirms the uniform bound on
\[
  \norm{y_h}_{L^{\frac{4\beta}{3\beta-3}}(\Omega\times(0,T); L^{2\beta})}.
\]
Subsequently, applying H\"older's inequality once more yields
\begin{align*}
  &\mathbb{E} \bigg[
    \int_0^T \norm{(P_hG-G_h)(t)}_{L^{2\beta'}}^p \norm{y_h(t)}_{2\beta}^p \, \mathrm{d}t
  \bigg] \\
  \leqslant{}
  & \norm{P_hG-G_h}_{L^{\frac{4p\beta}{4\beta-(3\beta-3)p}}(\Omega\times(0,T); L^{2\beta'})}^p
  \norm{y_h}_{L^{\frac{4\beta}{3\beta-3}}(\Omega\times(0,T);L^{2\beta})}^p \\
  \leqslant{}
  & c\norm{P_hG-G_h}_{L^{\frac{4p\beta}{4\beta-(3\beta-3)p}}(\Omega\times(0,T); L^{2\beta'})}^p.
\end{align*}
Using \cref{eq:G-Gh-PhG} with $ p $ replaced by $ \frac{4p\beta}{4\beta - (3\beta - 3)p} $ and $ q $ by $ 2\beta' $, we conclude that
\[
  \mathbb{E}\bigg[
    \int_0^T \norm{(P_hG - G_h)(t)}_{L^{2\beta'}}^p \norm{y_h(t)}_{L^{2\beta}}^p \, \mathrm{d}t
  \bigg]
  \leqslant c h^{2 - \frac{3p}{2}(1 - \frac{1}{\beta})}.
\]
Moreover, by \cref{thm:y-regu}, we have $ y \in L^{\frac{4\beta}{3(\beta - 1)}}(\Omega \times (0,T); L^{2\beta}) $, so a similar argument yields
\[
  \mathbb{E} \bigg[
    \int_0^T \norm{(P_hG - G_h)(t)}_{L^{2\beta'}}^p \norm{y(t)}_{L^{2\beta}}^p \, \mathrm{d}t
  \bigg]
  \leqslant c h^{2 - \frac{3p}{2}(1 - \frac{1}{\beta})}.
\]
Combining both estimates gives \cref{eq:PhG-Gh-y}.
The estimate \cref{eq:love} follows by a similar argument using \cref{eq:Deltah-Delta}.
This completes the proof.

\end{proof}

% \begin{lemma}
%   Assume that $ v \in L^2 $. Let $ y_h $ be the mild solution
%   of \cref{eq:yh}. Then
%   \begin{equation}
%     \label{eq:yh-stab-l2}
%     \norm{y_h}_{L^p(\Omega\times(0,T); L^{6p/(3p-4)})}
%     \leqslant c (1 + \norm{v}_{L^2}),
%     \quad \forall p \in (2,\infty).
%   \end{equation}
% \end{lemma}
% \begin{proof}
%   Similar to \cite[Equation (2.2)]{LiuWei2010}, we have
%   \[
%     \mathbb E \int_0^T \norm{y_h(t)}_{L^2}^{p-2}
%     \norm{y_h(t)}_{\dot H^{1,2}}^2
%     \, \mathrm{d}t \leqslant c(1 + \norm{v}_{L^2}^p),
%     \quad \forall p \in [2,\infty),
%   \]
%   which, combined with the fact that $ \dot H^{1,2} $ is continuously
%   embedded in $ L^6 $, implies
%   \[
%     \mathbb E \int_0^T \norm{y_h(t)}_{L^2}^{p-2}
%     \norm{y_h(t)}_{L^6}^2
%     \, \mathrm{d}t \leqslant c(1 + \norm{v}_{L^2}^p),
%     \quad \forall p \in [2,\infty).
%   \]
%   Hence, from the Young's inequality
%   \[
%     \norm{y_h(t)}_{L^{6p/(3p-4)}} \leqslant
%     c \norm{y_h(t)}_{L^2}^{1-2/p} \norm{y_h(t)}_{L^6}^{2/p},
%     \quad \forall p \in (2,\infty),
%   \]
%   we then obtain \cref{eq:yh-stab-l2}.
% \end{proof}

Finally, we are in a position to prove \cref{thm:conv,thm:conv-smooth}.

\medskip\noindent{\bf Proof of \cref{thm:conv}.}
Fix any $ p \in (2,\infty) $. Define
\begin{align}
  z(t)   & := y(t) - G(t), \quad 0 \leqslant t \leqslant T,
  \label{eq:z-def}                                              \\
  z_h(t) & := y_h(t) - G_h(t), \quad 0 \leqslant t \leqslant T,
  \label{eq:zh-def}
\end{align}
where $ G $ is defined by \cref{eq:G-def}, and $ G_h $ is given by
\begin{equation}
  \label{eq:Gh-def}
  G_h(t) := e^{t\Delta_h} P_h y_0, \quad t \in [0,T].
\end{equation}
From \cref{eq:model} and the fact that $ G'(t) = \Delta G(t) $ for $ t \in [0,T] $, we deduce that
$$
  \mathrm{d}z(t) = \big( \Delta z + z + G - y^3 \big)(t) \, \mathrm{d}t + F(y(t)) \, \mathrm{d}W_H(t),
  \quad 0 \leqslant t \leqslant T.
$$
Consequently,
\begin{align*}
  \mathrm{d}P_hz(t) = P_h(\Delta z + z + G - y^3)(t) \, \mathrm{d}t + P_hF(y(t)) \, \mathrm{d}W_H(t),
  \quad 0 \leqslant t \leqslant T.
\end{align*}
Similarly, from \cref{eq:yh} and the fact that $ G_h'(t) = \Delta_h G_h(t) $ for $ t \in [0,T] $, it follows that
\begin{align*}
  \mathrm{d}z_h(t) =
  (\Delta_h z_h + z_h + G_h - P_hy_h^3)(t) \, \mathrm{d}t +
  P_h F\big( y_h(t) \big) \, \mathrm{d}W_H(t),
  \quad 0 \leqslant t \leqslant T.
\end{align*}
Define $ e_h := z_h - P_h z $. Then we obtain
\begin{equation}
  \label{eq:eh}
  \begin{aligned}
    \mathrm{d}e_h(t)
    & = \Big(
      \Delta_h e_h + \Delta_hP_hz - P_h\Delta z + e_h + G_h-P_hG + P_h\big( y^3-y_h^3 \big)
    \Big)(t) \, \mathrm{d}t + {}       \\
    & \quad P_h \Big(
      F\big(y_h(t)\big) - F\big( y(t) \big)
      \Big) \, \mathrm{d}W_H(t), \quad 0 \leqslant t \leqslant T.
  \end{aligned}
\end{equation}
Now define $ \eta_h := e_h - \xi_h $, where $ \xi_h := S_{0,h}(\Delta_h P_hz - P_h\Delta z) $,
and consider the differential equation
$$
  \Big(\frac{\mathrm{d}}{\mathrm{d}t} - \Delta_h \Big)
  \xi_h(t) = (\Delta_hP_hz - P_h\Delta z)(t), \quad t \geqslant 0.
$$
Using \cref{eq:eh}, we conclude that
\begin{equation}
  \label{eq:etah}
  \begin{aligned}
    \mathrm{d}\eta_h(t)
      & = \Big(
        \Delta_h\eta_h + \xi_h + \eta_h + G_h-P_hG + P_h\big( y^3-y_h^3 \big)
      \Big)(t) \, \mathrm{d}t + {}          \\
      & \quad P_h \Big(
        F\big(y_h(t)\big) - F\big( y(t) \big)
        \Big) \, \mathrm{d}W_H(t), \quad 0 \leqslant t \leqslant T.
  \end{aligned}
\end{equation}
By applying Itô's formula to $ \norm{\eta_h(t)}_{L^2}^p $ and using the initial condition $\eta_h(0)=0$ together with the definition of $ F $ given by \cref{eq:F-def},
we deduce that for any $0 \leqslant t \leqslant T$,
\begin{align*}
  \mathbb E\norm{\eta_h(t)}_{L^2}^p
  & =
  p\mathbb E \biggl[
    \int_0^t \norm{\eta_h(s)}_{L^2}^{p-2}
  \int_\mathcal O \eta_h(s) \cdot \Big( \Delta_h \eta_h(s) + \xi_h(s) + \eta_h(s)  \\
  & \qquad\qquad\qquad\qquad\qquad {} + (G_h-P_hG)(s) +
\big( y^3  - y_h^3 \big)(s) \Big) \, \mathrm{d}x \, \mathrm{d}s
\biggr] \\
  & \qquad {} + \frac{p}2\,\mathbb E \biggl[
    \int_0^t \norm{\eta_h(s)}_{L^2}^{p-2}
    \sum_{k=0}^\infty \norm{P_h\big(f_k(y_h(s))-f_k(y(s))\big)}_{L^2}^2 \, \mathrm{d}s
  \biggr] \\
  & \qquad {}+ \frac{p(p-2)}2 \mathbb E\biggl[
    \int_0^t \norm{\eta_h(s)}_{L^2}^{p-4}
    \sum_{k=0}^\infty \Big(
      \int_\mathcal O \eta_h(s) \cdot \big(f_k(y_h(s)) - f_k(y(s)) \big) \,\mathrm{d}x
    \Big)^2 \, \mathrm{d}s
  \biggr] \\
  &= I_1 + I_2 + I_3 + I_4 + I_5 + I_6,
\end{align*}
where
\begin{align*}
  \allowdisplaybreaks
  I_1 & :=
  p \, \mathbb E \biggl[
    \int_0^t \norm{\eta_h(s)}_{L^2}^{p-2}
    \int_\mathcal O \eta_h(s) \cdot \Delta_h \eta_h(s) \,\mathrm{d}x\,\mathrm{d}s
  \biggr], \\
  I_2 & :=  p \, \mathbb E\biggl[
    \int_0^t \norm{\eta_h(s)}_{L^2}^{p-2}
    \int_\mathcal O \eta_h(s) \cdot  (\xi_h + \eta_h + G_h-P_hG)(s) \, \mathrm{d}x \, \mathrm{d}s
  \biggr], \\
    I_3 & := p\,\mathbb E\biggl[
      \int_0^t \norm{\eta_h(s)}_{L^2}^{p-2}
      \int_\mathcal O \eta_h(s) \cdot \big[ y^3  - (P_hy)^3 \big](s)  \, \mathrm{d}x \, \mathrm{d}s
    \biggr],  \\
      I_4 & := p\,\mathbb E\biggl[
        \int_0^t \norm{\eta_h(s)}_{L^2}^{p-2}
        \int_\mathcal O \eta_h(s) \cdot \big[ (P_hy)^3  - y_h^3 \big](s)  \, \mathrm{d}x \, \mathrm{d}s
      \biggr], \\
        I_5 & := \frac{p}2\,\mathbb E\biggl[
          \int_0^t \norm{\eta_h(s)}_{L^2}^{p-2}
          \sum_{k=0}^\infty \norm{P_h\big(f_k(y_h(s))-f_k(y(s))\big)}_{L^2}^2 \, \mathrm{d}s
        \biggr], \\
          I_6 & := \frac{p(p-2)}2 \mathbb E\biggl[
            \int_0^t \norm{\eta_h(s)}_{L^2}^{p-4}
            \sum_{k=0}^\infty \Big(
              \int_\mathcal O \eta_h(s) \cdot \big(f_k(y_h(s)) - f_k(y(s)) \big)\,\mathrm{d}x
            \Big)^2 \, \mathrm{d}s
          \biggr]. 
\end{align*}

Next, we estimate the terms $ I_1 $, $ I_2 $, $ I_3 $, and $ I_4 $.  
For the term $ I_1 $, by the definition of $ \Delta_h $, we obtain
\begin{equation}
  \label{eq:I1-bnd}
  I_1 = -p \mathbb E \biggl[ \int_0^t \|{\eta_h(s)}\|_{L^2}^{p-2} \|{\eta_h(s)}\|_{\dot H^{1,2}}^2
  \, \mathrm{d}s \biggr].
\end{equation}
Regarding $ I_2 $, applying Hölder's inequality, Young's inequality, and the bounds given in \cref{eq:G-Gh-PhG} and \cref{eq:Deltah-Delta}, we deduce that
\begin{align} 
  I_2 \leqslant ch^2 + c \int_0^t \|{\eta_h(s)}\|_{L^p(\Omega;L^2)}^p \, \mathrm{d}s.
  \label{eq:I2-bnd}
\end{align}
For $ I_3 $, using Hölder’s inequality, Young’s inequality, and the stability property of $ P_h $ stated in \cref{eq:Ph-stab}, a straightforward computation gives
\begin{align*}
  I_3 & \leqslant c \, \mathbb E\bigg[
    \int_0^t \norm{\eta_h(s)}_{L^2}^{p-1} \norm{(y-P_hy)(s)}_{L^4}
    \Big(
      \norm{y(s)}_{L^8}^2 + \norm{P_hy(s)}_{L^8}^2
    \Big) \, \mathrm{d}s
  \bigg] \\
      & \leqslant c \, \mathbb E\bigg[
        \int_0^t \norm{\eta_h(s)}_{L^2}^{p-1} \norm{(y-P_hy)(s)}_{L^4}
        \norm{y(s)}_{L^8}^2 \, \mathrm{d}s                                   
      \bigg] \\
      & \leqslant c \mathbb E\bigg[
        \int_0^t \left( \norm{\eta_h(s)}_{L^2}^p + \norm{(y-P_hy)(s)}_{L^4}^p
        \norm{y(s)}_{L^8}^{2p} \right) \, \mathrm{d}s
      \bigg].
  \end{align*}
  Using the estimate \cref{eq:y-Phy-y}, we then arrive at the bound
  \begin{equation}
    \label{eq:I3-bnd}
    I_3 \leqslant ch^{2-\epsilon} + c \int_0^t
    \norm{\eta_h(s)}_{L^p(\Omega;L^2)}^p
    \, \mathrm{d}s,
    \quad \forall \epsilon > 0.
  \end{equation}
  To estimate $ I_4 $, we proceed as follows.  
  Using the definitions in (\ref{eq:z-def}) and (\ref{eq:zh-def}), together with the identity $ z_h - P_hz = \xi_h + \eta_h $, we find that
  $$
  P_hy - y_h = -\xi_h - \eta_h + P_hG - G_h.
  $$
  Thus, for any $ s \in [0,t] $, we compute:
  \begin{align*}
  &\int_{\mathcal{O}} \eta_h(s) \left[(P_hy)^3 - y_h^3\right](s) \, \mathrm{d}x \\
    ={}& \int_{\mathcal{O}} \eta_h(s) (P_hy - y_h)(s)
    \big[(P_hy)^2 + (P_hy)y_h + y_h^2\big](s) \, \mathrm{d}x \\
    ={}&
    \int_{\mathcal{O}} \eta_h(s) (-\xi_h - \eta_h + P_hG - G_h)(s)
    \big[ (P_hy)^2 + (P_hy)y_h + y_h^2 \big](s) \, \mathrm{d}x \\
    \leqslant{}& c \int_{\mathcal{O}}
    \left[\xi_h^2(s) + (P_hG - G_h)^2(s)\right] \big[ (P_hy)^2 + y_h^2 \big](s) \, \mathrm{d}x.
  \end{align*}
  Hence, applying Hölder’s inequality, Young’s inequality, and the stability property of $ P_h $ from (\ref{eq:Ph-stab}), we derive the following estimate for $ I_4 $:
  \begin{align*}
    I_4
      &\leqslant c\mathbb{E}\int_0^t \bigg[
        \|\eta_h(s)\|_{L^2}^p +
        \left(\|\xi_h(s)\|^{p}_{L^{2\beta'}} 
        + \|(P_hG - G_h)(s)\|^{p}_{L^{2\beta'}}\right) \\
      & \qquad\qquad\qquad {} \times \Big(\|y(s)\|^{p}_{L^{2\beta}} +
    \|y_h(s)\|^{p}_{L^{2\beta}}\Big) \bigg] \, \mathrm{d}s,
    \quad \forall \beta \in \left(1, \frac{3p}{3p - 4}\right),
  \end{align*}
  where $ \beta' $ denotes the H\"older conjugate of $ \beta $.  
  Using the estimates \eqref{eq:PhG-Gh-y} and \eqref{eq:love}, this simplifies further to
  \begin{align*}
    I_4 \leqslant c\int_0^t \|\eta_h(s)\|^p_{L^p(\Omega;L^2)}
    \, \mathrm{d}s + ch^{2-\frac{3}{2}p\left(1-\frac{1}{\beta}\right)},
    \quad \forall \beta \in \left(1, \frac{3p}{3p - 4}\right).
  \end{align*}
  By choosing $ \beta $ sufficiently close to 1 within the interval $\left(1, \frac{3p}{3p - 4}\right)$, for any $ \epsilon > 0 $, we finally obtain
  \begin{equation}
    \label{eq:I4-bnd}
    I_4 \leqslant ch^{2-\epsilon} +
    c\int_0^t \|\eta_h(s)\|^p_{L^p(\Omega;L^2)} \, \mathrm{d}s,
    \quad \forall \epsilon > 0.
  \end{equation}

Now, we proceed to estimate $ I_5 + I_6 $.  
Using Hölder’s inequality and the stability property of $ P_h $ given in \cref{eq:Ph-stab}, we obtain the following bound:
\begin{align*}
  I_5 + I_6 \leqslant c \mathbb{E}\bigg[
    \int_0^t \norm{\eta_h(s)}_{L^2}^{p-2} \sum_{k=0}^\infty \norm{f_k(y_h(s)) - f_k(y(s))}_{L^2}^2 \, \mathrm{d}s
  \bigg].
\end{align*}
Invoking the condition \cref{eq:fny-l2}, we further deduce that
\begin{align*}
  I_5 + I_6 \leqslant c \mathbb{E} \bigg[
    \int_0^t \norm{\eta_h(s)}_{L^2}^{p-2} \norm{(y_h-y)(s)}_{L^2}^2 \, \mathrm{d}s
  \bigg].
\end{align*}
Applying Young’s inequality, we then have
\begin{align*}
  I_5 + I_6 \leqslant c \mathbb{E} \bigg[
    \int_0^t \left( \norm{\eta_h(s)}_{L^2}^p +
    \norm{(y_h-y)(s)}_{L^2}^p \right) \, \mathrm{d}s
  \bigg].
\end{align*}
Recalling the definitions in \cref{eq:z-def,eq:zh-def} and using the identity $ z_h - P_hz = \xi_h + \eta_h $, we infer that
\begin{equation}
  \label{eq:y-yh-decomp}
  (y_h - y)(s) = \xi_h(s) + \eta_h(s) - (z - P_hz)(s) - (G - G_h)(s), \quad \forall s \in [0,T].
\end{equation}
Combining this with the estimates \cref{eq:G-Gh-PhG}, \cref{eq:z-Phz}, and \cref{eq:Deltah-Delta}, we find that
$$
  \mathbb{E} \bigg[ \int_0^t \norm{(y_h - y)(s)}_{L^2}^p \, \mathrm{d}s \bigg]
  \leqslant ch^2 + c\int_0^t \norm{\eta_h(s)}_{L^p(\Omega;L^2)}^p \, \mathrm{d}s.
$$
As a result, we arrive at the following bound for $ I_5 + I_6 $:
\begin{equation}
  \label{eq:I5+I6-bnd}
  I_5 + I_6 \leqslant ch^2 + c\int_0^t
  \norm{\eta_h(s)}_{L^p(\Omega;L^2)}^p \, \mathrm{d}s.
\end{equation}

Finally, by combining the estimates for $I_1$, $I_2$, $I_3$, $I_4$, and $I_5 + I_6$,
as established in \cref{eq:I1-bnd,eq:I2-bnd,eq:I3-bnd,eq:I4-bnd,eq:I5+I6-bnd}, we obtain
\begin{align*}
& \|\eta_h(t)\|_{L^p(\Omega;L^2)}^p +
\mathbb E\bigg[\int_0^t \|\eta_h(s)\|_{L^2}^{p-2} \|\eta_h(s)\|_{\dot H_h^{1,2}}^2 \, \mathrm{d}s\bigg] \\
\leqslant{}
& ch^{2-\epsilon} + c \int_0^t
\|\eta_h(s)\|_{L^p(\Omega;L^2)}^p \, \mathrm{d}s,
\quad \forall t \in (0,T], \, \forall \epsilon > 0.
\end{align*}
An application of Gronwall’s inequality then yields
\[
  \max_{t \in [0,T]} \|\eta_h(t)\|_{L^p(\Omega;L^2)} \leqslant ch^{(2-\epsilon)/p},
  \quad \forall \epsilon > 0.
\]
It follows that
\begin{equation}
  \label{eq:etah-esti}
  \max_{t \in [0,T]} \|\eta_h(t)\|_{L^p(\Omega;L^2)} \leqslant ch^{2/p-\epsilon},
  \quad \forall \epsilon > 0.
\end{equation}
Moreover, for any $\epsilon > 0$ and $t \in (0,T]$, the following bounds hold:
\begin{alignat*}{2}
  & \norm{(z-P_hz)(t)}_{L^p(\Omega;L^2)} \leqslant
  ch^{1-\epsilon},   &\quad & \text{(see \cref{eq:Ph-conv,eq:y-G-C}),} \\
  & \norm{(G-G_h)(t)}_{L^2} \leqslant
  c h^{2/p}t^{-1/p}, &\quad & \text{(see \cref{eq:foo-conv4})}, \\
  & \norm{\xi_h(t)}_{L^p(\Omega;L^2)} \leqslant ch^{1-\epsilon}, &\quad & \text{(see \cref{eq:Deltah-Delta-2})}.
\end{alignat*}
Now, from the decomposition \cref{eq:y-yh-decomp}, we deduce that
\begin{align*}
\|(y-y_h)(t)\|_{L^p(\Omega;L^2)} 
& \leqslant \|(z-P_hz)(t)\|_{L^p(\Omega;L^2)} +
\|(G-G_h)(t)\|_{L^p(\Omega;L^2)} \\
& \qquad {} + \|\xi_h(t)\|_{L^p(\Omega;L^2)} + \|\eta_h(t)\|_{L^p(\Omega;L^2)}, \quad \forall t \in [0,T].
\end{align*}
Combining the above estimates, we conclude that for all $ t \in (0,T] $,
\begin{equation*}
  \norm{(y-y_h)(t)}_{L^p(\Omega;L^2)} \leqslant
  ch^{2/p-\epsilon} + ch^{1-\epsilon} + ch^{2/p}t^{-1/p},
  \quad \forall \epsilon > 0.
\end{equation*}
Since $ h^{1 - \epsilon} $ is absorbed into $ h^{2/p - \epsilon} $ for $ p > 2 $,
the desired error estimate \cref{eq:conv} follows immediately. This completes the proof of \cref{thm:conv}.

\hfill{$\blacksquare$}

\medskip\noindent{\bf Proof of \cref{thm:conv-smooth}.}
Let $ G $ and $ z $ be defined by \cref{eq:G-def,eq:z-def}, respectively.
Since the initial condition $ y_0 \in W_0^{1,\infty}(\mathcal{O}) \cap W^{2,\infty}(\mathcal{O}) $ implies that $ y_0 \in \dot{H}^{2,q} $
for all $ q \in (2,\infty) $, it follows from \cref{lem:etDelta} that $ G \in C([0,T]; \dot{H}^{2,q}) $ for all $ q \in (2,\infty) $.
Consequently, by adapting the argument in Step 3 of the proof of \cref{thm:y-regu}, we obtain the following improved regularity
properties for the processes $ y $ and $ z $:
\begin{align}
   & y \in L_{\mathbb F}^p(\Omega;C([0,T]; \dot H^{2-\epsilon,q})),
   && \forall p,q \in (2,\infty), \, \forall \epsilon \in (0,1);
  \label{eq:y-C-smooth}                                                                                                        \\
   & z \in L_{\mathbb F}^p(\Omega\times(0,T); \dot H^{2,q}),
   && \forall p,q \in (2,\infty);
  \label{eq:z-lplq-smooth}                                                                                                   \\
   & z \in L_{\mathbb F}^p(\Omega;C([0,T]; \dot H^{2-\epsilon,q})),
   && \forall p,q \in (2,\infty), \, \forall \epsilon \in (0,1).
   \label{eq:z-Clq-smooth}
\end{align}
Let $ G_h $ be defined as in \cref{eq:Gh-def}.  
We retain the notation introduced in the proof of \cref{thm:conv}:
\begin{align*}
  z_h &:= y_h - G_h, 
      && e_h := z_h - P_h z, \\
  \xi_h &:= S_{0,h}(\Delta_h P_h - P_h \Delta) z, 
        && \eta_h := e_h - \xi_h.
\end{align*}
The remainder of the proof is structured into the following three steps.

\medskip\noindent\textbf{Step 1.}
Let us present some key estimates that form the foundation of our subsequent analysis.
We begin by establishing the following estimates:
\begin{align}
  & \left\| z - P_hz \right\|_{L^p(\Omega \times (0, T); L^q)} \leqslant ch^2,
  \quad \forall p,q \in (2, \infty),
  \label{eq:z-Phz-smooth} \\
  & \left\| G - G_h \right\|_{C([0, T]; L^q)} +
  \left\| P_hG - G_h \right\|_{C([0, T]; L^q)} \leqslant ch^2, \quad \forall q \in (2, \infty).
  \label{eq:G-Gh-smooth}
\end{align}
Estimate \eqref{eq:z-Phz-smooth} is a direct consequence of \eqref{eq:Ph-conv} combined
with the regularity property \eqref{eq:z-lplq-smooth}.
Estimate \eqref{eq:G-Gh-smooth} follows from \eqref{eq:foo-conv4} and the initial condition
$ y_0 \in W_0^{1,\infty}(\mathcal O) \cap W^{2,\infty}(\mathcal O) $.
Given the regularity \cref{eq:z-lplq-smooth}, estimate \cref{eq:Deltah-Delta} improves to
\begin{equation}
  \label{eq:xih-smooth}
  \left\| \xi_h \right\|_{L^p(\Omega \times (0, T); L^q)} \leqslant c h^2,
  \quad \forall p,q \in (2, \infty),
\end{equation}
and estimate \cref{eq:Deltah-Delta-2} improves to
\begin{equation}
  \label{eq:xih-C-smooth}
  \norm{\xi_h}_{L^p(\Omega; C([0,T]; L^q))} \leqslant c h^{2 - 2/p - \epsilon}, \quad \forall p,q \in (2,\infty), \, \forall \epsilon \in (0,1).
\end{equation}
Following the derivation of estimate \cref{eq:etah-esti} for $\eta_h$, we obtain the improved bound
\begin{equation}
  \label{eq:etah-l2-smooth}
  \max_{0 \leqslant t \leqslant T} \left\| \eta_h(t) \right\|_{L^p(\Omega; L^2)}
  \leqslant c h^2, \quad \forall p \in (2, \infty).
\end{equation}
Therefore, combining estimates \cref{eq:xih-smooth,eq:etah-l2-smooth,eq:G-Gh-smooth,eq:z-Phz-smooth} via the decomposition \cref{eq:y-yh-decomp} yields the error estimate
\begin{equation}
  \label{eq:y-yh-lpCl2-smooth}
  \left\| y - y_h \right\|_{L^p(\Omega \times (0, T); L^2)} \leqslant c h^2, \quad \forall p \in (2, \infty).
\end{equation}
Furthermore, for any $ p \in (2,\infty) $, we deduce the following chain of inequalities:
\begin{align*}
  & \norm{y_h}_{L^p(\Omega\times(0,T);L^\infty)} \\
  ={}
  & \norm{P_hy - P_h(y-y_h)}_{L^p(\Omega\times(0,T);L^\infty)} \\
  \leqslant{}
  & \norm{P_hy}_{L^p(\Omega\times(0,T);L^\infty)} +
  \norm{P_h(y-y_h)}_{L^p(\Omega\times(0,T);L^\infty)} \\
 \stackrel{\mathrm{(i)}}{\leqslant}{}
  & \norm{P_hy}_{L^p(\Omega\times(0,T);L^\infty)} +
  ch^{-3/2}\norm{P_h(y-y_h)}_{L^p(\Omega\times(0,T);L^2)} \\
 \stackrel{\mathrm{(ii)}}{\leqslant}{}
  & c\norm{y}_{L^p(\Omega\times(0,T);L^\infty)} +
  ch^{-3/2}\norm{y-y_h}_{L^p(\Omega\times(0,T);L^2)} \\
 \stackrel{\mathrm{(iii)}}{\leqslant}{}
  & c\norm{y}_{L^p(\Omega\times(0,T);L^\infty)} + ch^{1/2},
\end{align*}
where in (i) we apply the inverse estimate \(\norm{v_h}_{L^\infty} \leqslant ch^{-3/2} \norm{v_h}_{L^2} \) for all \(v_h \in X_h\)
(see, e.g., \cite[Theorem~4.5.11]{Brenner2008});
in (ii) we utilize the stability property of $ P_h $ as stated in \cref{eq:Ph-stab};
and in (iii) we employ the error estimate \cref{eq:y-yh-lpCl2-smooth}.
By \cref{eq:y-C-smooth} and the continuous embedding $ \dot H^{2-\epsilon,q} \hookrightarrow L^\infty $
for sufficiently small $ \epsilon > 0 $,
we conclude $h$-uniform boundedness:
\begin{equation}
  \label{eq:yh-Linfty}
  \norm{y_h}_{L^p(\Omega\times(0,T);L^\infty)}
  \text{ is uniformly bounded with respect to \(h\) for all \(p\in (2,\infty)\)}.
\end{equation}

\medskip\noindent\textbf{Step 2.}  
We now prove the error estimate \cref{eq:y-yh-LpLq}. The proof comprises two parts, (a) and (b).

\textbf{Part (a).} We establish the estimate
\begin{equation}
  \label{eq:etah-L2}
  \| \eta_h \|_{L^p(\Omega \times (0,T); L^6)} \leqslant c h^2,
  \quad \forall p \in (2,\infty).
\end{equation}
Utilizing \eqref{eq:etah} and the definitions of the discrete convolution operators $S_{0,h}$ and $S_{1,h}$
given in \cref{eq:S0h-def,eq:S1h-def}, we derive the decomposition
\begin{align}
  \eta_h = II_1 + II_2,
  \label{eq:etah-decomp}
\end{align}
where
\begin{align*}
  II_1 &:= S_{0,h}\big( \xi_h + \eta_h + G_h - P_hG + P_h(y^3 - y_h^3) \big), \\
  II_2 &:= S_{1,h}P_h\big( F(y_h) - F(y) \big).
\end{align*}
Fix $p \in (2,\infty)$. For $II_1$, applying \cref{lem:Hh-Linfty}(i) and \cref{lem:S0h} yields
\begin{align*}
  \| II_1 \|_{L^p(\Omega \times (0,T); L^\infty)}
  &\leqslant c \| II_1 \|_{L^p(\Omega \times (0,T); \dot{H}_h^{2,5/3})} \\
  &\leqslant c \| \xi_h + \eta_h + G_h - P_hG + P_h(y^3 - y_h^3) \|_{L^p(\Omega \times (0,T); L^{5/3})} \\
  &\leqslant c \| \xi_h \|_{L^p(\Omega \times (0,T); L^{5/3})}
  + c \| \eta_h \|_{L^p(\Omega \times (0,T); L^{5/3})} \\
  & \quad{} + c \| G_h-P_hG \|_{L^p(\Omega \times (0,T); L^{5/3})} 
   + c \| P_h(y^3 - y_h^3) \|_{L^p(\Omega \times (0,T); L^{5/3})}.
\end{align*}
Employing the estimates \cref{eq:etah-l2-smooth,eq:G-Gh-smooth,eq:xih-smooth}
and the embedding $ L^q \hookrightarrow L^{5/3} $ for $ q > 5/3 $, we obtain
\begin{align*}
  \| II_1 \|_{L^p(\Omega \times (0,T); L^\infty)}
  \leqslant c h^2 + \| P_h(y^3 - y_h^3) \|_{L^p(\Omega \times (0,T); L^{5/3})}.
\end{align*}
Since $P_h$ is bounded on $L^{5/3}$ (see \eqref{eq:Ph-stab}), it follows that
\begin{align*}
  \| II_1 \|_{L^p(\Omega \times (0,T); L^\infty)}
  \leqslant c h^2 + \| y^3 - y_h^3 \|_{L^p(\Omega \times (0,T); L^{5/3})}.
\end{align*}
To estimate $\| y^3 - y_h^3 \|_{L^p(\Omega \times (0,T); L^{5/3})}$, we use the identity
\[
  y^3 - y_h^3 = (y - y_h)(y^2 + y y_h + y_h^2)
\]
and apply Hölder's inequality, leading to
\begin{align*}
  \| y^3 - y_h^3 \|_{L^p(\Omega \times (0,T); L^{5/3})}
  \leqslant c \| y - y_h \|_{L^{2p}(\Omega \times (0,T); L^2)}
  \Big( \| y \|_{L^{4p}(\Omega \times (0,T); L^{20})}^2 + \| y_h \|_{L^{4p}(\Omega \times (0,T); L^{20})}^2 \Big).
\end{align*}
By the regularity of $y$ (see \eqref{eq:y-C-smooth}), the stability of $y_h$ (see \eqref{eq:yh-Linfty}), and the error estimate \eqref{eq:y-yh-lpCl2-smooth}, we deduce
\begin{equation}
  \label{eq:y3-yh3-pre}
  \| y^3 - y_h^3 \|_{L^p(\Omega \times (0,T); L^{5/3})} \leqslant c h^2.
\end{equation}
Therefore,
\begin{equation}
  \label{eq:II1}
  \| II_1 \|_{L^p(\Omega \times (0,T); L^\infty)} \leqslant c h^2.
\end{equation}
For $II_2$, the continuous embedding $\dot H^{1,2} \hookrightarrow L^6$ and the isometric embedding of $\dot H_h^{1,2}$ into $\dot H^{1,2}$ imply
\[
  \| II_2 \|_{L^p(\Omega\times(0,T);L^6)}  \leqslant c \| II_2 \|_{L^p(\Omega\times(0,T);\dot H_h^{1,2})}.
\]
Hence, applying \cref{lem:S1h}, the ideal property of $\gamma$-radonifying operators \cite[Theorem 9.1.10]{HytonenWeis2017}, the boundedness of $P_h$ on $L^2$, the Lipschitz condition \eqref{eq:F-lips}, and the estimate \eqref{eq:y-yh-lpCl2-smooth}, we obtain
\begin{align*}
  \| II_2 \|_{L^p(\Omega\times(0,T);L^6)}
  &\leqslant c \| P_h\big(F(y_h) - F(y)\big) \|_{L^p(\Omega\times(0,T);\gamma(H,\dot H_h^{0,2}))} \\
  &\leqslant c \| F(y_h) - F(y) \|_{L^p(\Omega\times(0,T);\gamma(H,L^2))} \\
  &\leqslant c \| y - y_h \|_{L^p(\Omega\times(0,T);L^2)} \\
  &\leqslant c h^2.
\end{align*}
Combining \eqref{eq:II1} with this estimate via the decomposition \eqref{eq:etah-decomp} establishes \eqref{eq:etah-L2}.

\textbf{Part (b).}
We consider the decomposition of $ y_h - y $ given in \cref{eq:y-yh-decomp}, and combine it with the estimates \cref{eq:z-Phz-smooth}, \cref{eq:G-Gh-smooth}, \cref{eq:xih-smooth}, and \cref{eq:etah-L2} to establish \cref{eq:y-yh-LpLq} for the particular case $ p \in (2,\infty) $, $ q = 6 $.
Based on this, we now apply \cref{lem:Hh-Linfty}(ii) and adapt the same bounding strategy used for $ II_2 $ in part (a). This yields an improved estimate:
\begin{align*}
  \norm{II_2}_{L^p(\Omega\times(0,T);L^\infty)}
  & \leqslant c \norm{II_2}_{L^p(\Omega\times(0,T);\dot H_h^{1,6})} \\
  & \leqslant c \norm{P_h(F(y_h)-F(y))}_{L^p(\Omega\times(0,T);\gamma(H,\dot H_h^{0,6}))} \\
  & \leqslant c \norm{F(y_h)-F(y)}_{L^p(\Omega\times(0,T);\gamma(H,L^6))} \\
  & \leqslant c \norm{y_h - y}_{L^p(\Omega\times(0,T);L^6)} \\
  & \leqslant c h^2, \quad \forall p \in (2,\infty).
\end{align*}
Combining this refined bound with the estimate for $ II_1 $ given in \cref{eq:II1}, and recalling the decomposition \cref{eq:etah-decomp}, we deduce that
\begin{equation}
  \label{eq:etah-lplq}
  \norm{\eta_h}_{L^p(\Omega \times (0,T); L^\infty)} \leqslant c h^2, \quad \forall p \in (2,\infty).
\end{equation}
Finally, revisiting the decomposition \cref{eq:y-yh-decomp} and combining the estimates \cref{eq:z-Phz-smooth}, \cref{eq:G-Gh-smooth}, \cref{eq:xih-smooth}, and \cref{eq:etah-lplq}, we conclude the error estimate \cref{eq:y-yh-LpLq} for all $ p, q \in (2, \infty) $.

\medskip\noindent\textbf{Step 3.}
We now establish the error estimate \cref{eq:y-yh-LpLinfty}, proceeding in three parts, (a)–(c).

\textbf{Part (a).}
We aim to prove the inequality
\begin{equation}
  \label{eq:eh-lpC}
  \| \eta_h \|_{L^p(\Omega; C([0,T]; L^\infty))} \leqslant c h^2,
  \quad \forall p \in (4,\infty).
\end{equation}
To this end, fix $ p \in (4, \infty) $ and let $ q \in (6, \infty) $ be arbitrary.
Recall the decomposition \cref{eq:etah-decomp} for $ \eta_h $. By \cref{lem:Hh-Linfty}(ii) and \cref{lem:real-complex-discrete}, we have
\begin{align*}
  \| II_1 \|_{L^p(\Omega; C([0,T]; L^\infty))}
  &\leqslant \| II_1 \|_{L^p(\Omega; C([0,T]; \dot H_h^{2 - 3/p, q}))} \\
  &\leqslant \| II_1 \|_{L^p(\Omega; C([0,T]; (\dot H_h^{0,q}, \dot H_h^{2,q})_{1/2 - 1/p,p}))}.
\end{align*}
An application of \cref{lem:S0h} then gives
\begin{align*}
  \| II_1 \|_{L^p(\Omega; C([0,T]; L^\infty))}
  &\leqslant c \| \xi_h + \eta_h + G_h - P_h G + P_h(y^3 - y_h^3) \|_{L^p(\Omega \times (0,T); L^q)}.
\end{align*}
Using the estimates \cref{eq:G-Gh-smooth}, \cref{eq:xih-smooth}, and \cref{eq:etah-lplq}, we deduce
\begin{align*}
  \| II_1 \|_{L^p(\Omega; C([0,T]; L^\infty))}
  &\leqslant c h^2 + c \| P_h(y^3 - y_h^3) \|_{L^p(\Omega \times (0,T); L^q)}.
\end{align*}
Since the projection operator $ P_h $ is bounded on $ L^q $ (cf. \cref{eq:Ph-stab}), it follows that
\begin{align*}
  \| II_1 \|_{L^p(\Omega; C([0,T]; L^\infty))}
  &\leqslant c h^2 + c \| y^3 - y_h^3 \|_{L^p(\Omega \times (0,T); L^q)}.
\end{align*}
As in \cref{eq:y3-yh3-pre}, we use the identity
\[
  y^3 - y_h^3 = (y - y_h)(y^2 + y y_h + y_h^2)
\]
and apply Hölder’s inequality to obtain
\[
  \| y^3 - y_h^3 \|_{L^p(\Omega \times (0,T); L^q)} \leqslant c \| y - y_h \|_{L^{2p}(\Omega \times (0,T); L^{2q})}.
\]
Therefore,
\begin{align*}
  \| II_1 \|_{L^p(\Omega; C([0,T]; L^\infty))}
  &\leqslant c h^2 + c \| y - y_h \|_{L^{2p}(\Omega \times (0,T); L^{2q})}.
\end{align*}
Next, we estimate $ II_2 $ from \cref{eq:etah-decomp}.
Applying \cref{lem:real-complex-discrete} and \cref{lem:Hh-Linfty}(ii), we get
\begin{align*}
  \| II_2 \|_{L^p(\Omega; C([0,T]; L^\infty))}
  &\leqslant \| II_2 \|_{L^p(\Omega; C([0,T]; (\dot H_h^{0,q}, \dot H_h^{2,q})_{1/2 - 1/p,p}))} \\
  &\leqslant c \| P_h(F(y_h) - F(y)) \|_{L^p(\Omega \times (0,T); \gamma(H, \dot H_h^{0,q}))},
\end{align*}
where the second inequality follows from \cref{lem:S1h}. As in the estimate of $ II_2 $ in Step 2, Part (a), we then obtain
\[
  \| II_2 \|_{L^p(\Omega; C([0,T]; L^\infty))}
  \leqslant c \| y - y_h \|_{L^p(\Omega \times (0,T); L^q)}.
\]
Combining the bounds for $ II_1 $ and $ II_2 $, and recalling the decomposition \cref{eq:etah-decomp}, we arrive at
\begin{align*}
  \| \eta_h \|_{L^p(\Omega; C([0,T]; L^\infty))}
  &\leqslant c h^2 + c \| y - y_h \|_{L^p(\Omega \times (0,T); L^q)} + c \| y - y_h \|_{L^{2p}(\Omega \times (0,T); L^{2q})}.
\end{align*}
Now, since the error estimate \cref{eq:y-yh-LpLq} has already been established for all $ p, q \in (2,\infty) $, we conclude that the right-hand side is of order $ h^2 $. This completes the proof of \cref{eq:eh-lpC}.

\textbf{Part (b).}
We now prove the estimate
\begin{equation}
  \label{eq:y-yh-lpC-lq}
  \| y - y_h \|_{L^p(\Omega; C([0,T]; L^q))} \leqslant c h^{2 - 2/p - \varepsilon},
  \quad \forall p, q \in (2,\infty),\ \forall \varepsilon \in (0,1).
\end{equation}
By the approximation property \cref{eq:Ph-conv} and the regularity result \cref{eq:z-Clq-smooth}, we have
\begin{equation}
  \label{eq:z-Phz-lpC}
  \| z - P_h z \|_{L^p(\Omega; C([0,T]; L^q))} \leqslant c h^{2 - \varepsilon},
  \quad \forall p, q \in (2,\infty),\ \forall \varepsilon \in (0,1).
\end{equation}
Moreover, combining the approximation result \cref{eq:G-Gh-smooth} with the bound \cref{eq:eh-lpC}, we deduce
\begin{equation}
  \label{eq:G-Gh-eta}
  \| G - G_h + \eta_h \|_{L^p(\Omega; C([0,T]; L^q))} \leqslant c h^2,
  \quad \forall p, q \in (2,\infty).
\end{equation}
Using the decomposition \cref{eq:y-yh-decomp}, we write
\begin{align*}
  \| y - y_h \|_{L^p(\Omega; C([0,T]; L^q))}
  &\leqslant \| \xi_h \|_{L^p(\Omega; C([0,T]; L^q))}
  + \| z - P_h z \|_{L^p(\Omega; C([0,T]; L^q))} \\
  &\quad + \| G - G_h + \eta_h \|_{L^p(\Omega; C([0,T]; L^q))}.
\end{align*}
Substituting the bounds \cref{eq:z-Phz-lpC,eq:G-Gh-eta} into this inequality, we obtain
\begin{align*}
  \| y - y_h \|_{L^p(\Omega; C([0,T]; L^q))}
  &\leqslant c h^{2 - \varepsilon} + \| \xi_h \|_{L^p(\Omega; C([0,T]; L^q))},
  \quad \forall p, q \in (2,\infty),\ \forall \varepsilon \in (0,1).
\end{align*}
Finally, invoking the estimate \cref{eq:xih-C-smooth}, we conclude the desired result \cref{eq:y-yh-lpC-lq}.

\textbf{Part (c).} 
Let $\Pi_h$ denote the standard Lagrange interpolation operator onto the finite element space $X_h$.
Combining the regularity result \cref{eq:y-C-smooth} with standard interpolation theory yields
\begin{align*}
  \norm{y - \Pi_h y}_{L^p(\Omega; C([0,T]; L^q))} &\leqslant c h^{2-\epsilon},
  \quad \forall p,q \in (2,\infty), \, \forall \epsilon \in (0,1).
\end{align*}
The triangle inequality applied to this estimate and \cref{eq:y-yh-lpC-lq} then gives
\begin{align*}
  \norm{y_h - \Pi_h y}_{L^p(\Omega; C([0,T]; L^q))} &\leqslant c h^{2-2/p-\epsilon},
  \quad \forall p,q \in (2,\infty), \, \forall \epsilon \in (0,1).
\end{align*}
Applying the standard finite element inverse estimate (cf. \cite[Theorem~4.5.11]{Brenner2008})
\begin{equation*}
  \norm{v_h}_{L^\infty} \leqslant c h^{-3/q} \norm{v_h}_{L^q}, \quad \forall v_h \in X_h, \, q \in (2,\infty),
\end{equation*}
to the preceding bound implies
\begin{align*}
  \norm{y_h - \Pi_h y}_{L^p(\Omega; C([0,T]; L^\infty))} &\leqslant c h^{2 - 2/p - 3/q - \epsilon},
  \quad \forall p,q \in (2,\infty), \, \forall \epsilon \in (0,1).
\end{align*}
Furthermore, the regularity \cref{eq:y-C-smooth} combined with standard interpolation estimates yields
\begin{align*}
  \norm{y - \Pi_h y}_{L^p(\Omega; C([0,T]; L^\infty))} &\leqslant c h^{2 - 3/q - \epsilon},
  \quad \forall p \in (2,\infty), \, \forall q \in (3,\infty), \, \forall \epsilon \in (0,1).
\end{align*}
Applying the triangle inequality to these last two estimates now provides
\begin{align*}
  \norm{y - y_h}_{L^p(\Omega; C([0,T]; L^\infty))}
  &\leqslant c h^{2-2/p-3/q-\epsilon}, \quad \forall p,q \in (2,\infty), \, \forall \epsilon \in (0,1).
\end{align*}
Since $p, q \in (2,\infty)$ can be chosen arbitrarily large, the term $2/p + 3/q$ can be made arbitrarily small and is therefore absorbed into $\epsilon > 0$. This yields the desired error estimate \cref{eq:y-yh-LpLinfty}, completing the proof of \cref{thm:conv-smooth}.

\hfill{$\blacksquare$}

\section{Conclusions}
\label{sec:conclusion}
In this work, we have rigorously investigated the regularity properties of the three-dimensional stochastic Allen-Cahn equation with
multiplicative noise, in the presence of non-smooth initial data. For the spatial semidiscretization using the finite element method,
we have established a convergence rate of \( O(h^{2/p-\epsilon} + h^{2/p}t^{-1/p}) \) with non-smooth initial data.
Meanwhile, under the condition of smooth initial data, employing the discrete stochastic maximal \( L^p \)-regularity theory,
we have achieved a convergence rate of \( O(h^2) \) in the \( L^p(\Omega \times (0,T); L^q) \) norm,
along with a pathwise uniform convergence rate of \( O(h^{2-\epsilon}) \).

The theoretical results derived here lay a foundation for the numerical analysis of fully discrete approximations based on
the finite element method for the three-dimensional stochastic Allen-Cahn equation with multiplicative noise,
within the context of general spatial \( L^q \)-norms.

\appendix

\section{Proofs of \texorpdfstring{\cref{lem:Hh-Linfty,lem:etDeltah}}{}}
\label{sec:appendix}

The discrete Laplace operator $ \Delta_h $ has the following fundamental property; see \cite{Bakaev2002}.
\begin{lemma}
  \label{lem:Deltah-resolvent}
  For any $q \in (1,\infty)$ and $\theta \in (\pi/2, \pi)$, there exists a constant $c > 0$ independent of $h$, depending only on $\theta$, $q$, and the regularity constants of $\mathcal{K}_h$, such that
  \[
    \left\| (re^{i\omega} - \Delta_h)^{-1} \right\|_{\mathcal{L}(\dot{H}_h^{0,q})} \leqslant \frac{c}{1 + r},
     \quad \forall r \geqslant 0, \quad \forall \omega \in [-\theta, \theta].
  \]
\end{lemma}

\begin{lemma}
  \label{lem:Ph-stab-extend}
  For any $\alpha \in [0,2]$ and $q \in (1,\infty)$, the operator norm
  $ \|P_h\|_{\mathcal{L}(\dot H^{\alpha,q}, \dot H_h^{\alpha,q})} $
  is uniformly bounded with respect to $h$.
\end{lemma}
\begin{proof}
  Fix $q \in (1,\infty)$. By the estimates \cref{eq:Ph-conv,eq:Ritz-conv-H2}, and applying the triangle inequality, we obtain
  \[
    \norm{P_{h} - \Delta_{h}^{-1}P_{h}\Delta}_{\mathcal{L}(\dot H^{2,q}, L^{q})} \leqslant c h^{2}.
  \]
  Using the inverse estimate \eqref{eq:inverse}, it follows that
  \[
    \norm{P_{h} - \Delta_{h}^{-1}P_{h}\Delta}_{\mathcal{L}(\dot H^{2,q}, \dot H_{h}^{2,q})} \leqslant c.
  \]
  On the other hand, the stability estimate \eqref{eq:Ph-stab} gives
  \[
    \norm{\Delta_{h}^{-1}P_{h}\Delta}_{\mathcal{L}(\dot H^{2,q}, \dot H_{h}^{2,q})} = \norm{P_{h}\Delta}_{\mathcal{L}(\dot H^{2,q}, L^{q})} \leqslant c\norm{\Delta}_{\mathcal{L}(\dot H^{2,q}, L^{q})} = c,
  \]
  where we have used the fact that $\Delta: \dot H^{2,q} \to L^q$ is an isometric isomorphism.
  Combining the above bounds and applying the triangle inequality, we conclude that
  \[
    \norm{P_{h}}_{\mathcal{L}(\dot H^{2,q}, \dot H_{h}^{2,q})} \leqslant c.
  \]
  This proves the uniform boundedness for $\alpha = 2$.
  The case $\alpha = 0$ is already established by \eqref{eq:Ph-stab}.
  The general case $\alpha \in [0,2]$ follows from complex interpolation between the endpoint cases $\alpha = 0$
  and $\alpha = 2$, which completes the proof.
\end{proof}

By the approximation property of $ P_h $ in \cref{eq:Ph-conv} and the smoothing estimate for $ e^{t\Delta} $ in \cref{lem:etDelta},
a direct calculation yields the following lemma.
\begin{lemma}
  \label{lem:I-Ph-etDelta}
  For any $\alpha \in [0,2]$, $q \in (1,\infty)$, and $t > 0$, 
  \begin{equation}
    \label{eq:I-Ph-etDelta}
    \left\| (I - P_h)e^{t\Delta} \right\|_{\mathcal{L}(\dot{H}^{\alpha,q}, L^q)} 
    \leqslant c \min\left\{ h^\alpha,\, \frac{h^2}{t^{1 - \alpha/2}} \right\}.
  \end{equation}
\end{lemma}

% \begin{proof}
%   Let $v \in \dot{H}^{\alpha,q}$ be arbitrary. By \eqref{eq:Ph-conv} and the standard estimate (see \cref{lem:etDelta})
%   \[
%     \|e^{t\Delta}v\|_{\dot{H}^{2,q}} \leqslant c t^{\alpha/2 - 1} \|v\|_{\dot{H}^{\alpha,q}},
%   \]
%   we have
%   \[
%     \|(I - P_h)e^{t\Delta}v\|_{L^q} \leqslant c h^2 t^{\alpha/2 - 1} \|v\|_{\dot{H}^{\alpha,q}}.
%   \]
%   Conversely, the uniform bound $\sup_{t \geqslant 0}\|e^{t\Delta}\|_{\mathcal{L}(\dot{H}^{\alpha,q})} \leqslant c$ and \eqref{eq:Ph-conv} yield
%   \[
%     \|(I - P_h)e^{t\Delta}v\|_{L^q} \leqslant c h^\alpha \|v\|_{\dot{H}^{\alpha,q}}.
%   \]
%   Combining these estimates establishes \eqref{eq:I-Ph-etDelta}.
% \end{proof}

  % and (iii) follows from (i) and (ii) via complex interpolation.

\subsection{Proof of \texorpdfstring{\cref{lem:Hh-Linfty}}{}}
We only prove (ii), as (i) follows by similar arguments using \cref{eq:Ritz-conv-H2} instead of \cref{eq:Ritz}.
Fix $ \alpha \in (0,2) $ and $ q \in (3/\alpha,\infty) \cap [2,\infty) $, and let $ v_h \in X_h $ be arbitrary.  
By \cref{eq:Ritz}, we have
$$
  \|v_h - \Delta^{-1} \Delta_h v_h\|_{L^q}
  = \|-(I - \Delta_h^{-1}P_h\Delta)\Delta^{-1} \Delta_h v_h\|_{L^q}
  \leqslant c h^\alpha \|\Delta^{-1} \Delta_h v_h\|_{\dot H^{\alpha,q}}.
$$
Moreover, \cref{eq:Ph-conv} gives
$$
  \|(I - P_h)\Delta^{-1} \Delta_h v_h\|_{L^q} \leqslant c h^\alpha \|\Delta^{-1} \Delta_h v_h\|_{\dot H^{\alpha,q}}.
$$
Combining these two estimates via the triangle inequality yields
$$
  \|v_h - P_h \Delta^{-1} \Delta_h v_h\|_{L^q} \leqslant c h^\alpha \|\Delta^{-1} \Delta_h v_h\|_{\dot H^{\alpha,q}}.
$$
Applying a standard inverse estimate \cite[Theorem~4.5.11]{Brenner2008}, we deduce
$$
  \|v_h - P_h \Delta^{-1} \Delta_h v_h\|_{L^\infty} \leqslant c h^{\alpha - 3/q} \|\Delta^{-1} \Delta_h v_h\|_{\dot H^{\alpha,q}}.
$$
Since $ \alpha > 3/q $, it follows that
$$
  \|v_h - P_h \Delta^{-1} \Delta_h v_h\|_{L^\infty} \leqslant c \|\Delta^{-1} \Delta_h v_h\|_{\dot H^{\alpha,q}}.
$$
The Sobolev embedding $ \dot H^{\alpha,q} \hookrightarrow L^\infty $, which holds due to $ \alpha q > 3 $, together with the $ L^\infty $-stability of $ P_h $ (see \cref{eq:Ph-stab}), implies
$$
  \|P_h \Delta^{-1} \Delta_h v_h\|_{L^\infty} \leqslant c \|\Delta^{-1} \Delta_h v_h\|_{\dot H^{\alpha,q}}.
$$
Applying the triangle inequality once more, we obtain
$$
  \|v_h\|_{L^\infty} \leqslant c \|\Delta^{-1} \Delta_h v_h\|_{\dot H^{\alpha,q}} = c \|\Delta_h v_h\|_{\dot H^{\alpha-2,q}}.
$$
To estimate $ \|\Delta_h v_h\|_{\dot H^{\alpha-2,q}} $, we use the dual characterization
$$
  \|\Delta_h v_h\|_{\dot H^{\alpha-2,q}} = \sup_{\substack{\varphi \in \dot H^{2-\alpha,q'} \\ \varphi \neq 0}}
  \frac{\langle \Delta_h v_h, \varphi \rangle}{\|\varphi\|_{\dot H^{2-\alpha,q'}}}.
$$
Since $ \Delta_h v_h \in X_h $, we may replace $ \varphi $ with $ P_h \varphi $, yielding
$$
  \|\Delta_h v_h\|_{\dot H^{\alpha-2,q}} = \sup_{\substack{\varphi \in \dot H^{2-\alpha,q'} \\ \varphi \neq 0}}
  \frac{\int_{\mathcal{O}} (-\Delta_h)^{\alpha/2} v_h \cdot (-\Delta_h)^{1-\alpha/2} P_h \varphi \,\mathrm{d}x}{\|\varphi\|_{\dot H^{2-\alpha,q'}}}.
$$
By Hölder's inequality and \cref{lem:Ph-stab-extend}, we deduce
$$
  \|\Delta_h v_h\|_{\dot H^{\alpha-2,q}} \leqslant \sup_{\substack{\varphi \in \dot H^{2-\alpha,q'} \\ \varphi \neq 0}}
  \frac{\|v_h\|_{\dot H_h^{\alpha,q}} \|P_h \varphi\|_{\dot H_h^{2-\alpha,q'}}}{\|\varphi\|_{\dot H^{2-\alpha,q'}}}
  \leqslant c \|v_h\|_{\dot H_h^{\alpha,q}}.
$$
Combining the above estimates, we conclude
$$
  \|v_h\|_{L^\infty} \leqslant c \|v_h\|_{\dot H_h^{\alpha,q}},
$$
which completes the proof of (ii), and hence the lemma.

  \hfill{$\blacksquare$}

\subsection{Proof of \texorpdfstring{\cref{lem:etDeltah}}{}}
Property (i) follows directly from Proposition 2.5 and Equation (2.128) in \cite{Yagi2010}, in light of \cref{lem:Deltah-resolvent}. To prove property (ii), let $\alpha \in [0,2]$, $q \in (1,\infty)$, and $t > 0$ be arbitrary. By \cref{lem:I-Ph-etDelta}, the desired estimate for the second term on the left-hand side of \cref{eq:foo-conv4} is already established. It therefore remains to show that
\begin{equation}
  \label{eq:etDeltahPh-PhetDelta}
  \norm{ e^{t\Delta_h} P_h - P_h e^{t\Delta} }_{\mathcal{L}(\dot{H}^{\alpha,q}, L^q)}
  \leqslant c \min\left\{ h^{\alpha}, \frac{h^{2}}{t^{1 - \alpha/2}} \right\}.
\end{equation}
This estimate is well known and can be derived using arguments found in \cite[Theorem 6.10, Lemma 20.7]{Thomee2006}.
 For the sake of completeness, we provide a concise proof of \eqref{eq:etDeltahPh-PhetDelta} in two steps.

\medskip\noindent\textbf{Step 1.} We establish the estimate
\begin{equation}
\left\| e^{t\Delta_h}P_hv - P_h e^{t\Delta}v \right\|_{L^q} \leqslant c h^2 \|v\|_{\dot{H}^{2,q}},
\quad \forall v \in \dot H^{2,q}.
\label{eq:zhibo-1}
\end{equation}
Fix $ v \in \dot{H}^{2,q} $. Utilizing the Laplace inversion formula, we represent the semigroups as:
\begin{align*}
e^{t\Delta}v &= \frac{1}{2\pi i} \int_\Gamma e^{t\lambda} \lambda^{-1} (\lambda - \Delta)^{-1} \Delta v  \,\mathrm{d}\lambda + v, \\
e^{t\Delta_h}P_h v &= \frac{1}{2\pi i} \int_\Gamma e^{t\lambda} \lambda^{-1} (\lambda - \Delta_h)^{-1} \Delta_h P_h v  \,\mathrm{d}\lambda + P_h v,
\end{align*}
where the contour $\Gamma$, oriented such that the negative real axis lies to its left, is defined by
$$
\Gamma = \left\{ r e^{\pm i\theta} \mid \, r \in [t^{-1}, \infty) \right\} \cup \left\{ t^{-1} e^{i\omega} | \, \omega \in [-\theta, \theta] \right\},
$$
with $\theta \in (\pi/2, \pi)$ and $i$ denoting the imaginary unit.
Subtracting these expressions yields:
\begin{equation}
e^{t\Delta_h}P_h v - P_h e^{t\Delta} v = \frac{1}{2\pi i} \int_\Gamma e^{t\lambda} \lambda^{-1} \left[ (\lambda - \Delta_h)^{-1} \Delta_h P_h - P_h (\lambda - \Delta)^{-1} \Delta \right] v  \,\mathrm{d}\lambda.
\label{eq:sum}
\end{equation}
A direct computation establishes the identity
\begin{align*}
  \left[ (\lambda - \Delta_h)^{-1} \Delta_h P_h - P_h (\lambda - \Delta)^{-1} \Delta \right] v
= (\lambda - \Delta_h)^{-1} \Delta_h \left( P_h \Delta^{-1} - \Delta_h^{-1} P_h \right) \lambda (\lambda - \Delta)^{-1} \Delta v.
\end{align*}
This implies
\begin{align*}
& \left\| \left( (\lambda - \Delta_h)^{-1} \Delta_h P_h - P_h (\lambda - \Delta)^{-1} \Delta \right) v \right\|_{L^q} \\
\leqslant{}
&\left\| (\lambda - \Delta_h)^{-1} \Delta_h \right\|_{\mathcal{L}(\dot{H}_h^{0,q})}
\left\| P_h \Delta^{-1} - \Delta_h^{-1} P_h \right\|_{\mathcal{L}(L^q)}
\left\| \lambda (\lambda - \Delta)^{-1} \right\|_{\mathcal{L}(L^q)} \|v\|_{\dot{H}^{2,q}} \\
\leqslant{}
& ch^2 \left\| \lambda (\lambda - \Delta)^{-1} \right\|_{\mathcal{L}(L^q)} \|v\|_{\dot{H}^{2,q}} \\
\leqslant{}
& c h^2 \|v\|_{\dot{H}^{2,q}},
\end{align*}
where the second inequality follows from Lemma~\ref{lem:Deltah-resolvent}
and \cref{eq:Ritz-conv-Halpha} with $ \alpha=2 $,
and the third inequality uses the uniform boundedness of $\|\lambda (\lambda - \Delta)^{-1}\|_{\mathcal{L}(L^q)}$ for $\lambda \in \Gamma$.
Applying this bound to the integral representation in \eqref{eq:sum}, we deduce
\begin{align*}
  \norm{e^{t\Delta_h}P_hv - P_he^{t\Delta}v}_{L^q}
  \leqslant ch^2 \norm{v}_{\dot H^{2,q}}
  \int_{\Gamma} \big| e^{t\lambda}\big| \, \frac{| \mathrm{d}\lambda |}{|\lambda|}
  \leqslant ch^2 \norm{v}_{\dot H^{2,q}},
\end{align*}
which establishes \eqref{eq:zhibo-1}.

\medskip\noindent\textbf{Step 2.} For arbitrary \( v \in L^q \), the Laplace inversion formula gives
\begin{align*}
  e^{t\Delta}v &= \frac{1}{2\pi i} \int_{\Gamma} e^{t\lambda} (\lambda - \Delta)^{-1} v  \,\mathrm{d}\lambda, \\
  e^{t\Delta_h} P_h v &= \frac{1}{2\pi i} \int_{\Gamma} e^{t\lambda} (\lambda - \Delta_h)^{-1} P_h v  \,\mathrm{d}\lambda,
\end{align*}
where \( \Gamma \) is the contour defined in Step 1. Therefore, the difference is
\begin{align*}
  e^{t\Delta_h} P_h v - P_h e^{t\Delta} v
  &= \frac{1}{2\pi i} \int_{\Gamma} e^{t\lambda} \bigl[
    (\lambda - \Delta_h)^{-1} P_h - P_h (\lambda - \Delta)^{-1}
  \bigr] v  \,\mathrm{d}\lambda.
\end{align*}
Mimicking the derivation of \eqref{eq:zhibo-1} leads to the low-regularity estimate
\begin{equation}
  \label{eq:low_regularity}
  \norm{ e^{t\Delta_h} P_h v - P_h e^{t\Delta} v }_{L^q}
  \leqslant c h^{2} t^{-1} \norm{v}_{L^q},
  \quad \forall v \in L^q.
\end{equation}
Applying complex interpolation between \eqref{eq:zhibo-1} and \eqref{eq:low_regularity} yields
\begin{equation}
  \label{eq:interp_time}
  \norm{ e^{t\Delta_h} P_h v - P_h e^{t\Delta} v }_{L^q}
  \leqslant c h^{2} t^{\alpha/2 - 1} \norm{v}_{\dot{H}^{\alpha,q}},
  \quad \forall v \in \dot{H}^{\alpha,q}.
\end{equation}
Simultaneously, the following stability estimates hold for all \( v \in L^q \):
\begin{align*}
  \norm{ e^{t\Delta_h} P_h v }_{L^q}
  &\leqslant c \norm{P_h v}_{L^q} \leqslant c \norm{v}_{L^q}, \\
  \norm{ P_h e^{t\Delta} v }_{L^q}
  &\leqslant c \norm{ e^{t\Delta} v }_{L^q} \leqslant c \norm{v}_{L^q}.
\end{align*}
These rely on the boundedness of \( P_h \) on \( L^q \) (see \cref{eq:Ph-stab}), the boundedness of \( e^{t\Delta_h} \) on \( L^q \) (see \cref{eq:etDeltah}), and the boundedness of \( e^{t\Delta} \) on \( L^q \) (see \cref{lem:etDelta}). Consequently,
\begin{equation}
  \label{eq:uniform_bound}
  \norm{ e^{t\Delta_h} P_h v - P_h e^{t\Delta} v }_{L^q}
  \leqslant c \norm{v}_{L^q},
  \quad \forall v \in L^q.
\end{equation}
Complex interpolation between \eqref{eq:zhibo-1} and \eqref{eq:uniform_bound} then gives
\begin{equation}
  \label{eq:interp_space}
  \norm{ e^{t\Delta_h} P_h v - P_h e^{t\Delta} v }_{L^q}
  \leqslant c h^{\alpha} \norm{v}_{\dot{H}^{\alpha,q}},
  \quad \forall v \in \dot{H}^{\alpha,q}.
\end{equation}
Combining estimates \eqref{eq:interp_time} and \eqref{eq:interp_space}, we conclude that
\begin{equation*}
  \norm{ e^{t\Delta_h} P_h v - P_h e^{t\Delta} v }_{L^q}
  \leqslant c \min\left\{ h^{\alpha},\, h^{2} t^{\alpha/2 - 1} \right\} \norm{v}_{\dot{H}^{\alpha,q}},
  \quad \forall v \in \dot{H}^{\alpha,q},
\end{equation*}
which establishes \eqref{eq:etDeltahPh-PhetDelta}. This completes the proof of property (ii).

\hfill{$\blacksquare$}

\end{document}